\documentclass[reqno,11pt]{article}
\usepackage{graphicx, amsmath, amsthm}
\usepackage{amssymb}

\def\pathconnected{path-connected}

\input{stochbif.sty}

\def\bigexpecin#1#2{\E^{\mskip1.5mu #1}\bigl\{\mskip-1mu#2\mskip-1mu\bigr\}}   
\def\Norm#1{\norm{#1}_2}

\begin{document}   


\title{The Eyring--Kramers law \\
for potentials with nonquadratic saddles}
\author{Nils Berglund and Barbara Gentz}
\date{}   

\maketitle

\begin{abstract}
\noindent
The Eyring--Kramers law describes the mean transition time of an overdamped
Brownian particle between local minima in a potential landscape. In the
weak-noise limit, the transition time is to leading order exponential in  the
potential difference to overcome. This exponential is corrected by a prefactor
which depends on the principal curvatures of the potential at the starting
minimum and at the highest saddle crossed by an optimal transition path. The
Eyring--Kramers law, however, does not hold whenever one or more of these
principal
curvatures vanishes, since it would predict a vanishing or infinite transition
time. We derive the correct prefactor up to multiplicative errors that tend to
one in the zero-noise limit. As an illustration, we discuss the case of a
symmetric pitchfork bifurcation, in which the prefactor can be expressed in
terms of modified Bessel functions, as well as bifurcations with two vanishing
eigenvalues. The corresponding transition times are studied in a full
neighbourhood of the bifurcation point. These results extend work by Bovier,
Eckhoff,
Gayrard and Klein~\cite{BEGK},
who rigorously analysed the case of quadratic saddles, using
methods from potential theory. 
\end{abstract}

\leftline{\small{\it Date.\/} April 7, 2009. {\it Revised.\/} October 29, 2009.}
\leftline{\small 2000 {\it Mathematical Subject Classification.\/} 
60J45, 31C15 (primary), 60J60, 37H20 (secondary)}

\tableofcontents
\newpage


\section{Introduction}
\label{sec_intro}

Consider the stochastic differential equation 
\begin{equation}
\label{intro1}
\6x_t = -\nabla V(x_t)\,\6t + \sqrt{2\eps} \,\6W_t\;,
\end{equation}
where $V:\R^d\to\R$ is a confining potential. The Eyring--Kramers law
(\cite{Eyring,Kramers}) describes the expected transition time $\tau$ between
potential minima in the small-noise limit $\eps\to0$. In the one-dimensional
case ($d=1$), it has the following form. Assume $x$ and $y$ are quadratic local
minima of $V$, separated by a unique quadratic local maximum $z$. Then the
expected transition time from $x$ to $y$ satisfies
\begin{equation}
\label{intro2}
\bigexpecin{x}{\tau} \simeq
\frac{2\pi}{\sqrt{V''(x)\abs{V''(z)}}} \e^{[V(z)-V(x)]/\eps}\;.
\end{equation}
In the multidimensional case ($d\geqs2$), assume the local minima are separated
by a unique saddle $z$, which is such that the Hessian $\hessian{V(z)}$ admits a
single negative eigenvalue $\lambda_1(z)$, while all other eigenvalues are
strictly positive. Then the analogue of~\eqref{intro2} reads 
\begin{equation}
\label{intro3}
\bigexpecin{x}{\tau} \simeq
\frac{2\pi}{\abs{\lambda_1(z)}}
\sqrt{\frac{\det(\hessian{V}(z))}{\det(\hessian{V}(x))}}
\e^{[V(z)-V(x)]/\eps}\;.
\end{equation}
This expression has been generalised to situations where there are several
alternative saddles allowing to go from $x$ to $y$, and to potentials with more
than two minima. 

A long time has elapsed between the first presentation of the
formula~\eqref{intro3} by Eyring \cite{Eyring} and Kramers~\cite{Kramers} and
its rigorous mathematical proof (including a precise definition of what the
symbol \lq\lq$\simeq$\rq\rq\ in~\eqref{intro3} actually means). The exponential
asymptotics were proved to be correct by Wentzell and Freidlin in the early
Seventies, using the theory of large deviations~\cite{VF69,VF70,FW}. While being
very flexible, and allowing to study more general than gradient systems
like~\eqref{intro1}, large deviations do not allow to obtain the prefactor of
the transition time. An alternative approach is based on the fact that mean
transition times obey certain elliptic partial differential equations, whose
solutions can be approximated by WKB-theory (for a recent survey of these
methods, see~\cite{Kolokoltsov00}). This approach provides formal asymptotic
series expansions in $\eps$, whose justification is, however, a difficult
problem of analysis. A framework for such a rigorous justification is provided
by microlocal analysis, which was primarily developed by Helffer and Sj\"ostrand
to solve quantum mechanical tunnelling problems in the semiclassical
limit~\cite{HelfferSjostrand1,HelfferSjostrand2,HelfferSjostrand3,
HelfferSjostrand4}. Unfortunately, it turns out that when translated into terms
of semiclassical analysis, the problem of proving the Eyring--Kramers formula
becomes a particularly intricate one, known as \lq\lq tunnelling through
non-resonant wells\rq\rq. The first mathematically rigorous proof
of~\eqref{intro3} in arbitrary dimension (and its generalisations to more than
two wells) was obtained by Bovier, Eckhoff, Gayrard and Klein~\cite{BEGK}, using
a different approach based on potential theory and a variational principle.
In~\cite{BEGK}, the Eyring--Kramers law is shown to hold with $a\simeq b$
meaning $a=b(1+\Order{\eps^{1/2}\abs{\log\eps}})$. Finally, a full
asymptotic expansion of the prefactor in powers of $\eps$ was proved to hold
in~\cite{HelfferKleinNier04,HelfferNier05}, using again analytical methods.

In this work, we are concerned with the case where the determinant of one of the
Hessian matrices vanishes. In such a case, the expression~\eqref{intro3} either
diverges or goes to zero, which is obviously absurd. It seems reasonable (as
has been pointed out, e.g., in~\cite{Stein04}) that one has to take
into account higher-order terms of the Taylor expansion of the potential at the
stationary points when estimating the transition time. Of course, cases with
degenerate Hessian are in a sense not generic, so why should we care about this
situation at all? The answer is that as soon as the potential depends on a
parameter, degenerate stationary points are bound to occur, most noteably at
bifurcation points, i.e., where the number of saddles varies as the parameter
changes. See, for instance,~\cite{BFG06a,BFG06b} for an analysis of a naturally
arising parameter-dependent system displaying a series of symmetry-breaking
bifurcations. For this particular system, an analysis of the subexponential
asymptotics of metastable transition times in the synchronisation regime has
been given recently in~\cite{BarretBovier2007}, with a careful control of the
dimension-dependence of the error terms.

In order to study sharp asymptotics of expected transition times, we rely on the
potential-theoretic approach developed in~\cite{BEGK,BGK}. In particular, the
expected transition time can be expressed in terms of so-called Newtonian
capacities between sets, which can in turn be estimated by a variational
principle involving Dirichlet forms. The main new aspect of the present work is
that we estimate capacities in cases involving nonquadratic saddles.

In the non-degenerate case, saddles are easy to define: they are stationary
points at which the Hessian has exactly one strictly  negative eigenvalue, all
other eigenvalues being strictly positive. When the determinant of the Hessian
vanishes, the situation is not so simple, since the nature of the stationary
point depends on higher-order terms in the Taylor expansion. We thus start, in
Section~\ref{sec_saddles}, by defining and classifying saddles in degenerate
cases. In Section~\ref{sec_capacity}, we estimate capacities for the most
generic cases, which then allows us to derive expressions for the expected
transition times. In Section~\ref{sec_bif}, 
we extend these results to a number of bifurcation scenarios arising in typical 
applications, that is, we consider parameter-dependent potentials for parameter
values in a full neighbourhood of a critical parameter value yielding
non-quadratic saddles. Section~\ref{sec_proofs} contains the proofs of the main
results. 


\bigskip
\noindent
{\bf Acknowledgements:}
We would like to thank Bastien Fernandez for helpful discussions and Anton
Bovier for providing a preliminary version of~\cite{BarretBovier2007}.
BG thanks the MAPMO, Orl\'eans, and NB the CRC 701 {\it Spectral  
Structures and Topological Methods in Mathematics\/} at the  
University of Bielefeld, for kind hospitality.
Financial support by the French Ministry of Research, by way of the  
{\it \mbox{Action} Concert\'ee Incitative (ACI) Jeunes Chercheurs,
Mod\'elisation stochastique de syst\`emes hors \'equilibre\/}, and the  
German Research Council (DFG), by way of the CRC 701 {\it 
Spectral Structures and Topological Methods in Mathematics\/}, is  
gratefully acknowledged. 


\section{Classification of nonquadratic saddles}
\label{sec_saddles}

We consider a continuous, confining potential $V:\R^d\to\R$, bounded below by
some $a_{0}\in \R$ and having exponentially tight level sets, that is, 
\begin{equation}
\label{saddle0}
\int_{\setsuch{x\in\R^d}{V(x)\geqs a}} \e^{-V(x)/\eps} \,\6x 
\leqs C(a)\e^{-a/\eps}
\qquad
\forall a\ge a_{0}\;,
\end{equation}
with $C(a)$ bounded above and uniform in $\eps\leqs 1$. 
We start by giving a topological definition of saddles, before classifying
saddles for sufficiently differentiable potentials $V$.


\subsection{Topological definition of saddles}
\label{ssec_stop}

We start by introducing the notion of a gate between two sets $A$ and $B$.
Roughly speaking, a \defwd{gate}\/ is a set that cannot be avoided by those
paths
going from $A$ to $B$ which stay as low as possible in the potential
landscape. \defwd{Saddles}\/ will then be defined as particular points in
gates. 

It is useful to introduce some terminology and notations:

\begin{itemiz}
\item For $x, y\in\R^d$, we denote by $\gamma: x\to y$ a \defwd{path}\/ from $x$
to $y$, that is, a continuous function $\gamma\colon [0,1]\to\R^d$ such that
$\gamma(0)=x$ and $\gamma(1)=y$.

\item	The \defwd{communication height}\/ between $x$ and $y$ is the highest
potential value no path leading from $x$ to $y$ can avoid reaching, even when
staying as low as possible, i.e., 
\begin{equation}
\label{stop1}
\Vbar(x,y) = \inf_{\gamma\colon x\to y} \sup_{t\in[0,1]} V(\gamma(t))\;.
\end{equation}
Note that $\Vbar(x,y)\geqs V(x)\vee V(y)$, with equality holding, for instance,
in cases where $x$ and $y$ are \lq\lq on the same side of a mountain
slope\rq\rq.

\item	The communication height between two sets $A, B\subset\R^d$ is given by 
\begin{equation}
\label{stop2}
\Vbar(A,B) = \inf_{x\in A, y\in B} \Vbar(x,y)\;.
\end{equation}
We denote by $\cG(A,B)=\setsuch{z\in\R^d}{V(z)=\Vbar(A,B)}$ the level set of
$\Vbar(A,B)$.

\item	The \defwd{set of minimal paths}\/ from $A$ to $B$ is 
\begin{equation}
\label{stop3}
\cP(A,B) = \Bigsetsuchvert{\gamma\colon x\to y}{x\in A, y\in B, \sup_{t\in[0,1]}
V(\gamma(t)) = \Vbar(A,B)}\;.
\end{equation}

\end{itemiz}

The following definition is taken from~\cite{BEGK}.

\begin{definition}
\label{def_stop1}
A \defwd{gate} $G(A,B)$ is a minimal subset of $\cG(A,B)$ such that all minimal
paths $\gamma\in\cP(A,B)$ must intersect $G(A,B)$.
\end{definition}

\begin{figure}
\centerline{\includegraphics*[clip=true,width=140mm]{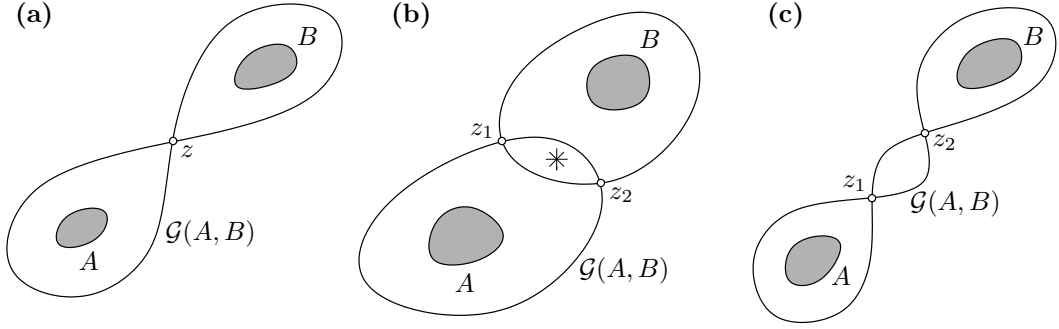}}
 \figtext{ 
	\writefig	0.5	4.5	{\bf (a)}
	\writefig	1.35	1.2	$A$
	\writefig	4.25	4.2	$B$
	\writefig	2.7	2.7	$z$
	\writefig	2.5	1.6	$\cG(A,B)$
	\writefig	5.5	4.5	{\bf (b)}
	\writefig	6.35	0.9	$A$
	\writefig	8.8	4.15	$B$
	\writefig	6.55	3.0	$z_1$
	\writefig	8.4	2.15	$z_2$
	\writefig	8.0	1.1	$\cG(A,B)$
	\writefig	10.5	4.5	{\bf (c)}
	\writefig	11.35	1.0	$A$
	\writefig	13.9	4.2	$B$
	\writefig	11.5	2.3	$z_1$
	\writefig	12.7	2.85	$z_2$
	\writefig	12.4	2.0	$\cG(A,B)$
 }
 \vspace{2mm}
\caption[]{Examples of potentials and gates. {\bf (a)} $G(A,B)=\set{z}$. 
{\bf (b)} $G(A,B)=\set{z_1,z_2}$. 
{\bf (c)} $G(A,B)=\set{z_1}$ or $\set{z_2}$.
Here curves show level lines, shaded areas indicate potential wells and the star
marks a potential maximum. 
}
\label{fig_gates}
\end{figure}

Let us consider some examples in dimension $d=2$ (\figref{fig_gates}):
\begin{itemiz}
\item	In uninteresting cases, e.g.\ for $A$ and $B$ on the same side of a
slope, the gate $G(A,B)$ is a subset of $A\cup B$. We will not be concerned
with such cases. 
\item	If on the way from $A$ to $B$, one has to cross one \lq\lq mountain
pass\rq\rq\ $z$ which is higher than all other passes, then $G(A,B)=\set{z}$
(\figref{fig_gates}a).
\item	If there are several passes at communication height $\Vbar(A,B)$
between $A$ and $B$, between which one can choose, then the gate $G(A,B)$ is
the union of these passes (\figref{fig_gates}b).
\item	If when going from $A$ to $B$, one has to cross several passes in a
row, all at communication height $\Vbar(A,B)$, then the gate $G(A,B)$ is not
uniquely defined: any of the passes will form a gate (\figref{fig_gates}c). 
\item	If $A$ and $B$ are separated by a ridge of constant altitude
$\Vbar(A,B)$, then the whole ridge is the gate $G(A,B)$. 
\item	If the potential contains a flat part separating $A$ from $B$, at
height $\Vbar(A,B)$, then any curve in this part separating the two sets is a
gate.
\end{itemiz}

We now proceed to defining saddles as particular cases of isolated points in
gates. However, the definition should be independent of the choice of sets $A$
and $B$. In order to do this, we start by introducing notions of valleys
(cf.~\figref{fig_valleys}):

\begin{itemiz}
\item	The \defwd{closed valley} of a point $x\in\R^d$ is the set 
\begin{equation}
\label{stop4}
\CV{x} = \bigsetsuch{y\in\R^d}{\Vbar(y,x)=V(x)}\;.
\end{equation}
It is straightforward to check that $\CV{x}$ is closed and \pathconnected. 

\item	The \defwd{open valley} of a point $x\in\R^d$ is the set 
\begin{equation}
\label{stop5}
\OV{x} = \bigsetsuch{y\in\CV{x}}{V(y)<V(x)}\;.
\end{equation}
It is again easy to check that $\OV{x}$ is open. Note however that if the
potential contains horizontal parts, then $\CV{x}$ need not be the closure of
$\OV{x}$ (\figref{fig_valleys}c). Also note that $\OV{x}$ need not be
\pathconnected\ (\figref{fig_valleys}b). We will use this fact to define a
saddle.
\end{itemiz}

Let $\cB_\eps(x)=\setsuch{y\in\R^d}{\Norm{y-x}<\eps}$ denote the open ball
of radius $\eps$, centred in $x$.

\begin{definition}
\label{def_stop2}
A \defwd{saddle} is a point $z\in\R^d$ such that there exists $\eps>0$ for
which 
\begin{enum}
\item	$\OV{z}\cap\cB_\eps(z)$ is non-empty and not \pathconnected.
\item	$(\OV{z}\cup\set{z})\cap\cB_\eps(z)$ is \pathconnected.
\end{enum}
\end{definition}

\begin{figure}
\centerline{\includegraphics*[clip=true,width=140mm]{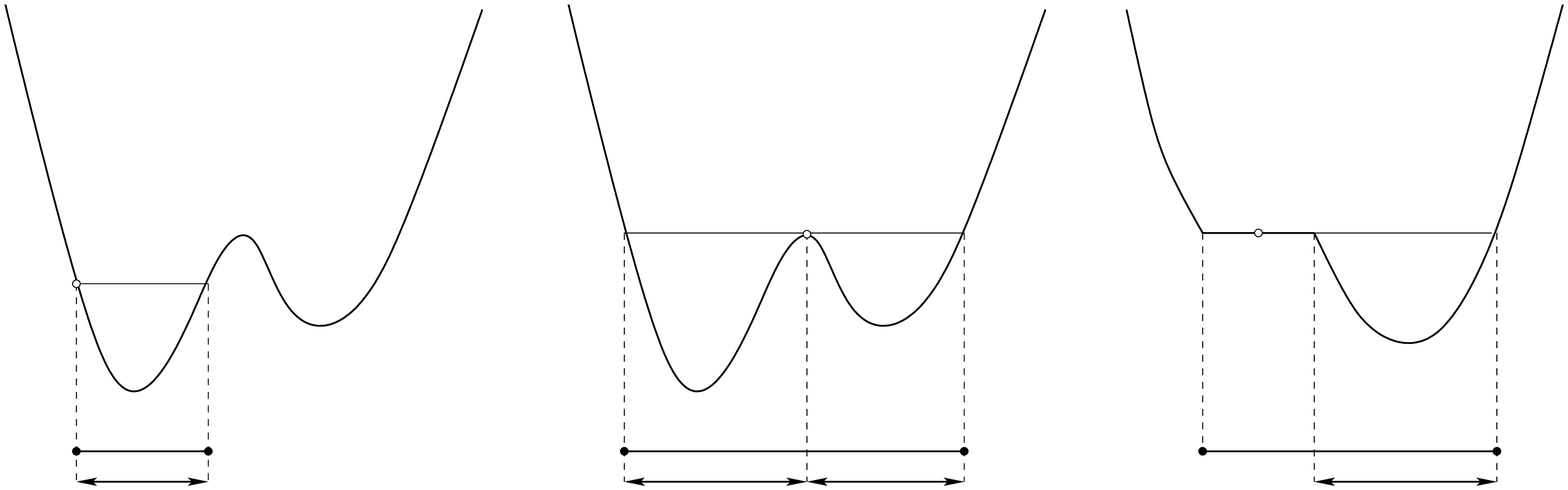}}
 \figtext{ 
	\writefig	0.55	4.6	{\bf (a)}
	\writefig	2.4	0.75	$\CV{x}$
	\writefig	2.4	0.35	$\OV{x}$
	\writefig	1.1	2.4	$x$
	\writefig	5.55	4.6	{\bf (b)}
	\writefig	9.1	0.75	$\CV{x}$
	\writefig	9.1	0.35	$\OV{x}$
	\writefig	7.5	2.9	$x$
	\writefig	10.55	4.6	{\bf (c)}
	\writefig	13.9	0.75	$\CV{x}$
	\writefig	13.9	0.35	$\OV{x}$
	\writefig	11.5	2.9	$x$
 }
 \vspace{4mm}
\caption[]{Examples of potentials and valleys. 
 In case {\bf (b)}, $x$ is a saddle. 
}
\label{fig_valleys}
\end{figure}

The link between saddles and gates is made clear by the following two results.

\begin{prop}
\label{prop_stop1}
Let $z$ be a saddle. Assume $\OV{z}$ is not \pathconnected,\footnote{We need
to make this assumption globally, in order to rule out situations where $z$ is
not the lowest saddle between two domains.} and let $A$ and
$B$ belong to different \pathconnected\ components of
$\OV{z}\cap\cB_\eps(z)$. Then $z\in G(A,B)$.
\end{prop}
\begin{proof}
Choose points $x\in A$ and $y\in B$ and a path $\gamma\colon A\to B$. Since $A$
and $B$ belong to different \pathconnected\ components of $\OV{z}$, the path
$\gamma$
must leave $\OV{z}$, which implies $\sup_{t\in[0,1]}V(\gamma(t))\geqs V(z)$.
Since $(\OV{z}\cup\set{z})\cap\cB_\eps(z)$ is \pathconnected, we can find a path
$\gamma\colon A\to B$ staying all the time in this set, and thus for this path,
$\sup_{t\in[0,1]}V(\gamma(t))= V(z)$. As a consequence, the communication height
$\Vbar(x,y)$ equals $V(z)$, i.e., $\gamma$ belongs to the set $\cP(A,B)$ of
minimal paths. Since $\OV{z}$ is not \pathconnected, we have found a path
$\gamma\in\cP(A,B)$ which must contain $z$, and thus $z\in G(A,B)$.
\end{proof}

\begin{prop}
\label{prop_stop2}
Let $A$ and $B$ be two disjoint sets, and let $z\in G(A,B)$. Assume that $z$ is
isolated in the sense that there exists $\eps>0$ such that 
$\cB_\eps^*(z)\defby\cB_\eps(z)\setminus\set{z}$ 
is disjoint from the union of all gates $G(A,B)$ between $A$ and
$B$. Then $z$ is a saddle. 
\end{prop}
\begin{proof}
Consider the set 
\begin{equation}
\label{stop6}
D = \bigcup_{\gamma\in\cP(A,B)} \bigcup_{t\in[0,1]} \gamma(t)
\end{equation}
of all points contained in minimal paths from $A$ to $B$. We claim that
$D=\CV{z}$. 

On one hand, if $x\in D$ then there exists a minimal path from $A$ to $B$
containing $x$. We follow this path backwards from $x$ to $A$. Then there is a
(possibly different) minimal path leading from the first path's endpoint in $A$
through $z$ to $B$. By gluing together these paths, we obtain a minimal path
connecting $x$ and $z$. This path never exceeds the potential value $V(z)$,
which proves $\Vbar(x,z)=V(z)$. Thus $x\in\CV{z}$, and $D\subset\CV{z}$
follows. 

On the other hand, pick $y\in\CV{z}$. Then there is a path
$\gamma_1\colon y\to z$ along which the potential does not exceed $V(z)$.
Inserting this path (twice, going back and forth) in a minimal path
$\gamma\in\cP(A,B)$ containing $z$, we get another minimal path from $A$ to $B$,
containing $y$. This proves $y\in D$, and thus the inverse inclusion
$\CV{z}\subset D$.

Now pick $x\in A$ and $y\in B$. There must exist a minimal path $\gamma\colon
x\to y$, containing $z$, with the property that $V$ is strictly smaller than
$V(z)$ on $\gamma([0,1])\cap\cB_\eps^*(z)$, since otherwise we
would contradict the assumption that $z$ be isolated. We can thus pick $x'$ on
$\gamma$ between $x$ and $z$ and $y'$ on $\gamma$ between $z$ and $y$ such that
$V(x')<V(z)$ and $V(y')<V(z)$. Hence we have $x',y'\in\OV{z}$ and any minimal
path from $x'$ to $y'$ staying in $\cB_\eps(z)$ has to cross $z\notin\OV{z}$.
This shows that $\OV{z}\cap\cB_\eps(z)$ is non-empty and not \pathconnected.
Finally, take any $x, y\in \OV{z}\cap\cB_\eps(z)\subset D$. Then we can connect
them by a path $\gamma\ni z$, and making $\eps$ small enough we may assume that
$V$ is strictly smaller than $V(z)$ on
$\gamma([0,1])\cap\cB_\eps^*(z)$, i.e.,
$\gamma([0,1])\setminus\set{z} \subset \OV{z}$. This proves that
$(\OV{z}\cup\set{z})\cap\cB_\eps(z)$ is \pathconnected.
\end{proof}


\subsection{Classification of saddles for differentiable potentials}
\label{ssec_sdiff}

Let us show that for sufficiently smooth potentials, our definition of saddles
is consistent with the usual definition of nondegenerate saddles. Then we will
start classifying degenerate saddles.

\begin{prop}
Let $V$ be of class $\cC^1$, and let $z$ be a saddle. Then $z$ is a stationary
point of $V$, i.e., $\nabla V(z)=0$. 
\end{prop}
\begin{proof}
Suppose, to the contrary, that $\nabla V(z)\neq 0$. We may assume $z=0$ and
$V(z)=0$. Choose local coordinates in which $\nabla V(0)=(a,0,\dots,0)$ with
$a>0$. By the implicit-function theorem, there exists a differentiable function
$h:\R^{d-1}\to\R^d$ such that all solutions of the equation $V(x)=0$ in a small
ball $\cB_\eps(0)$ are of the form $x_1=h(x_2,\dots,x_d)$. Furthermore,
$V(\eps,0,\dots,0)=a\eps+\order{\eps}$ is positive for $\eps>0$ and negative
for $\eps<0$. By continuity, $V(x)$ is positive for $x_1>h(x_2,\dots,x_d)$
and negative for $x_1<h(x_2,\dots,x_d)$, showing that $\OV{z}\cap\cB_\eps(0)$
is \pathconnected. Hence $z$ is not a saddle.
\end{proof}

\begin{prop}
\label{prop_sdiff2}
Assume $V$ is of class $\cC^2$, and let $z$ be a saddle. Then 
\begin{enum}
\item	the Hessian $\hessian{V(z)}$ has at least one eigenvalue smaller or
equal than $0$;
\item	the Hessian $\hessian{V(z)}$ has at most one eigenvalue strictly smaller
than $0$.
\end{enum} 
\end{prop}
\begin{proof}
Denote the eigenvalues of $\hessian{V(z)}$ by
$\lambda_1\leqs\lambda_2\leqs\dots\leqs\lambda_d$. We may again assume that
$z=0$ and $V(0)=0$, and choose a basis in which the Hessian is diagonal. Then 
\begin{equation}
\label{sdiff1}
V(x) = \frac12 \sum_{i=1}^d \lambda_i x_i^2 + \order{\Norm{x}^2}\;.
\end{equation}
\begin{enum}
\item	Assume, to the contrary that $\lambda_1>0$. Then $V>0$ near $z=0$, so
that
$\OV{z}=\emptyset$, and $z=0$ is not a saddle.  

\item	Suppose, to the contrary, that $\lambda_1\leqs\lambda_2<0$, and fix a
small $\delta>0$. Since 
\begin{equation}
\label{sdiff2}
V(x) = -\frac12\abs{\lambda_1}x_1^2 - \frac12\abs{\lambda_2} x_2^2 +
\frac12\sum_{i=3}^d \lambda_i
x_i^2 + \order{\Norm{x}^2}\;,
\end{equation}
we can find an $\eps=\eps(\delta)\in(0, \delta)$ such that for any fixed
$(x_3,\dots,x_d)$ of length less than $\eps$, the set
$\setsuch{(x_1,x_2)}{x_1^2+x_2^2<\delta^2, V(x)<0}$ 
is \pathconnected\ (topologically, it is an annulus). This implies that
$\setsuch{(x_1,\dots, x_d)}{x_1^2+x_2^2<\delta^2, V(x)<0,
\Norm{(x_3,\dots,x_d)}<\eps}$ 
is also \pathconnected. 
Hence $\OV{z}\cap\cB_\eps(z)$ is \pathconnected. 
\qed
\end{enum}
\renewcommand{\qed}{}
\end{proof}

\begin{prop}
\label{prop_sdiff3}
Assume $V$ is of class $\cC^2$, and let $z$ be a nondegenerate stationary point,
i.e.\ such that $\det(\hessian{V}(z))\neq0$. Then $z$ is a saddle if and only
if $\hessian{V}(z)$ has exactly one strictly negative eigenvalue.
\end{prop}

\begin{proof}
Denote the eigenvalues of $\hessian{V(z)}$ by
$\lambda_1\leqs\lambda_2\leqs\dots\leqs\lambda_d$. If $z$ is a saddle, then the
previous result implies that $\lambda_1<0<\lambda_2$. Conversely, if
$\lambda_1<0<\lambda_2$, we have 
\begin{equation}
\label{sdiff3}
V(x) = -\frac12 \abs{\lambda_1}x_1^2 + \frac12 \sum_{i=2}^d \lambda_i
x_i^2 + \order{\Norm{x}^2}\;.
\end{equation}
Thus for fixed small $(x_2,\dots,x_d)$, the set
$\setsuch{x_1}{\abs{x_1}<\eps, V(x)<0}$ is not \pathconnected, as it does not
contain $0$. Thus $\OV{z}\cap\cB_\eps(z)$ is not \pathconnected\ (it is
topologically
the interior of a double cone).
However, for $x_1=0$, adding the origin makes the set
\pathconnected, so that $(\OV{z}\cap\cB_\eps(z))\cup\set{z}$ is
\pathconnected.
\end{proof}

We can now classify all candidates for saddles in the following way. Let
$\lambda_1\leqs\lambda_2\leqs\dots\leqs\lambda_d$ be the eigenvalues of the
Hessian $\hessian{V(z)}$ of a stationary point $z$, arranged in increasing
order. Then the following cases may occur:
\begin{enum}
\item	$\lambda_1<0$: 
\begin{enum}
\item	$\lambda_2>0$: $z$ is a nondegenerate saddle.
\item	$\lambda_2=0$:
\begin{enum}
\item	$\lambda_3>0$: $z$ is a singularity of codimension $1$.
\item	$\lambda_3=0$:
\begin{enum}
\item	$\lambda_4>0$: $z$ is a singularity of codimension $2$.
\item	$\lambda_4=0$: $z$ is a singularity of codimension larger than $2$.
\end{enum}
\end{enum}
\end{enum}
\item	$\lambda_1=0$: 
\begin{enum}
\item	$\lambda_2>0$: $z$ is a singularity of codimension $1$.
\item	$\lambda_2=0$:
\begin{enum}
\item	$\lambda_3>0$: $z$ is a singularity of codimension $2$.
\item	$\lambda_3=0$: $z$ is a singularity of codimension larger than $2$.
\end{enum}
\end{enum}
\end{enum}

One can of course push further the classification, including all singularities
up to codimension $d$.


\subsection{Singularities of codimension 1}
\label{ssec_scodim1}

We assume in this subsection that the potential $V$ is of class $\cC^4$ and that
$z$ is a stationary point of $V$ with the Hessian $\hessian{V(z)}$ having one
vanishing eigenvalue. We may assume $z=0$ and $V(z)=0$. According to
Proposition~\ref{prop_sdiff2}, there are two cases to be considered: 
\begin{enum}
\item	$\lambda_1<0$, $\lambda_2=0$ and $0<\lambda_3\leqs\dots\leqs\lambda_d$.
\item	$\lambda_1=0$ and $0<\lambda_2\leqs\dots\leqs\lambda_d$.
\end{enum}
It will be convenient for the purpose of this subsection to relabel the first
two eigenvalues in such a way that $\lambda_1=0$, while $\lambda_2\neq 0$ can be
positive or negative and to choose a basis in which
$\hessian{V(z)}=\diag(0,\lambda_2,\dots,\lambda_d)$. For potentials $V$ of class
$\cC^r$ and $i_1,\dots,i_r\in\set{1,\dots,d}$, we introduce the notation
\begin{equation}
\label{scod1-1}
V_{i_1\dots i_r} = \frac{1}{n_1!\dots n_d!}
\frac{\partial}{\partial x_{i_1}}
\dots
\frac{\partial}{\partial x_{i_r}}
V(z)\;,
\qquad
n_j=\#\setsuch{k}{i_k=j}\;,
\end{equation}
with the convention that the $i_k$'s are always in increasing order. 

Under the above assumptions, the potential admits a Taylor expansion of the
form 
\begin{equation}
\label{scod1-2}
V(x) = \frac12\sum_{i=2}^d \lambda_i x_i^2 
+ \sum_{1\leqs i\leqs j\leqs k\leqs d} V_{ijk}x_ix_jx_k
+ \sum_{1\leqs i\leqs j\leqs k\leqs l\leqs d} V_{ijkl}x_ix_jx_kx_l
+ \order{\Norm{x}^4}\;,
\end{equation}
and the theory of normal forms allows to simplify this expression. Indeed, there
exists a change of variables $x=y+g(y)$, with $g$ a polynomial function,
such that
the potential expressed in the new variables has as few as possible terms of low
order in its Taylor expansion. In general, only a few so-called resonant terms
cannot be eliminated, and are thus essential for describing the local dynamics.

\begin{prop}
\label{prop_scodim1}
There exists a polynomial change of variables $x=y+g(y)$, where $g$ is a
polynomial with terms of degree $2$ and $3$, such that 
\begin{equation}
\label{scod1-3}
V(y+g(y)) = \frac12\sum_{i=2}^d \lambda_i y_i^2
+ C_3 y_1^3 + C_4 y_1^4 
+ \order{\Norm{y}^4}\;,
\end{equation}
where 
\begin{equation}
\label{scod1-4}
C_3 = V_{111}\;, \qquad
C_4 = V_{1111} - \frac12 \sum_{j=2}^d
\frac{V_{11j}^2}{\lambda_j}\;.
\end{equation}
\end{prop}

The proof uses standard normal form theory and will be given in
Appendix~\ref{app_nf}. Note that if $V$ is of class $\cC^5$, then
\eqref{scod1-3} holds with $ \order{\Norm{y}^4}$ replaced by 
$\Order{\Norm{y}^5}$. 
Let us now apply the result to derive an easy to verify necessary condition for
a point $z$ to be a saddle. 

\begin{cor}\hfill
\label{cor_scodim1}
\begin{enum}
\item 	Assume $\lambda_2<0$. Then the point $z$ is 
\begin{itemiz}
\item	a saddle if $C_3=0$ and $C_4>0$;
\item	not a saddle if $C_3\neq 0$ or $C_4<0$.
\end{itemiz}
\item 	Assume $\lambda_2>0$. Then the point $z$ is 
\begin{itemiz}
\item	a saddle if $C_3=0$ and $C_4<0$;
\item	not a saddle if $C_3\neq 0$ or $C_4>0$.
\end{itemiz}
\end{enum}
\end{cor}
\begin{proof}
Consider first the case $C_3\neq 0$. For simplicity, let us restrict to
$d=2$. In a neighbourhood of $z=0$, any solution to the equation $V(y)=0$ must
satisfy 
\begin{equation}
\label{scod1-5:1a}
y_2^{2} = -\frac{2C_3}{\lambda_2}y_1^{3} -
\frac{2C_4}{\lambda_2}y_1^{4} + \order{\Norm{y}^4}\;.
\end{equation}
Thus, solutions exist for  $y_{1}$ with $y_{1}C_3/\lambda_2<0$. Plugging the
ansatz
\begin{equation}
\label{scod1-5:1}
y_2 = \pm \sqrt{-\frac{2C_3}{\lambda_2}y_1^{3}} \,\bigbrak{1+r_2(y_1)}
\end{equation}
into the relation $V(y)=0$, dividing by $y_1^3$ and applying the
implicit-function theorem to the pair $(r_2,y_1)$ in the resulting equation
shows that there is a unique curve through the origin on which the potential
vanishes. Since for $y_2=0$, the potential has the same sign as $y_1$, we
conclude that $0$ is not a saddle. Now just note that the proof is similar in 
dimension $d>2$.

Consider next the case $C_3=0$ and $\lambda_2C_4>0$. If
$\lambda_2>0$, one sees that the origin is a local minimum. If $\lambda_2<0$,
then for fixed $(y_3,\dots,y_d)$ and sufficiently small $\eps$, the set
$\setsuch{(y_1,y_2)}{y_1^2+y_2^2<\eps^2, V(y)<0}$ is \pathconnected\ 
(topologically, it is an annulus). Hence $\OV{0}\cup\cB_\eps(0)$ is
\pathconnected. 

Finally, if $C_3=0$ and $\lambda_2C_4<0$, then either the 
set $\setsuch{y_1}{V(y)<0}$ for fixed $(y_2,\dots,y_d)$ or the set
$\setsuch{y_2}{V(y)<0}$ for fixed $(y_1,y_3,\dots,y_d)$ is not
\pathconnected, so that the valley of $0$ is locally split into two
disconnected components, joined at the origin.
\end{proof}

\begin{remark}
\label{rem_scod1}
The normal-form transformation $x\mapsto x+g(x)$ can also be applied when
$\lambda_1\neq0$. The result is exactly the same normal form as
in~\eqref{scod1-3}, except that there is an additional term $\frac12\lambda_1
y_1^2$. This observation is useful as it allows to study the system with a
unique
transformation of variables in a full neighbourhood of the bifurcation point.
\end{remark}

\begin{remark}
\label{rem_scod2}
One easily checks that if $V$ is of class $\cC^r$, one can construct
higher-order normal forms by eliminating all terms which are not of the form
$y_1^k$ with $k\leqs r$. In other words, there exists a polynomial $g(y)$ 
with terms of degree between $2$ and $r-1$
such that
\begin{equation}
\label{scod1-7}
V(y+g(y)) = \frac12\sum_{i=2}^d \lambda_i y_i^2
+ C_3 y_1^3 + C_4 y_1^4 
+ \dots + C_r y_1^r
+ \order{\Norm{y}^r}\;. 
\end{equation}
In general, however, there is no simple expression of the coefficients of the
normal form in terms of the original Taylor coefficients of $V$. 
Note that for $z=0$ to be a saddle, the first index $q$ such that $C_{q}\ne0$
has to be even with $\lambda_{2}C_{q}<0$.
\end{remark}


\subsection{Singularities of codimension 2}
\label{ssec_scodim2}

In this subsection, we shall assume that the potential $V$ is of class $\cC^{r}$
for some $r\geqs3$ and that $z$ is a stationary point of $V$ with the Hessian
$\hessian{V(z)}$ having two vanishing eigenvalues. We may
assume $z=0$ and $V(z)=0$. According to Proposition~\ref{prop_sdiff2}, there are
two cases to be considered: 
\begin{enum}
\item	$\lambda_1<0$, $\lambda_2=\lambda_3=0$ and
$0<\lambda_4\leqs\dots\leqs\lambda_d$.
\item	$\lambda_1=\lambda_2=0$ and $0<\lambda_3\leqs\dots\leqs\lambda_d$.
\end{enum}
For the purpose of this subsection, it will be convenient to relabel the first
three eigenvalues in such a way that $\lambda_1=\lambda_2=0$, while
$\lambda_3\neq 0$ can be positive or negative. We choose a basis in which
$\hessian{V(z)}=\diag(0,0,\lambda_3,\dots,\lambda_d)$. 
The potential thus admits a Taylor expansion of the form 
\begin{equation}
\label{scod2-2}
V(x) = \frac12\sum_{i=3}^d \lambda_i x_i^2 
+ \sum_{1\leqs i\leqs j\leqs k\leqs d} V_{ijk}x_ix_jx_k
+ \sum_{1\leqs i\leqs j\leqs k\leqs l\leqs d} V_{ijkl}x_ix_jx_kx_l
+ \dots\;.
\end{equation}
The theory of normal forms immediately yields the following result. 

\begin{prop}
\label{prop_scodim2}
There exists a polynomial change of variables $x=y+g(y)$, where $g$ is a
polynomial with terms of degree between $2$ and $r$, such that 
\begin{equation}
\label{scod2-3}
V(y+g(y)) = \frac12\sum_{i=3}^d \lambda_i y_i^2
+\sum_{k=3}^r V_k(y_1,y_2)  
+ \order{\Norm{y}^{r}}\;,
\end{equation}
where each $V_k$ is a homogeneous polynomial of order $k$, i.e.,
\begin{align}
\nonumber
V_3(y_1,y_2) &= V_{111}y_1^3 + V_{112}y_1^2y_2 + V_{122}y_1y_2^2 +
V_{222}y_2^3\;,\\
V_4(y_1,y_2) &= V_{1111}y_1^4 + V_{1112}y_1^3y_2 + V_{1122}y_1^2y_2^2 +
V_{1222}y_1y_2^3 + V_{2222}y_2^4\;,
\label{scod2-4}
\end{align}
and similarly for the higher-order terms.
\end{prop}

The proof uses standard normal form theory and follows along the lines of the
proof of Proposition~\ref{prop_scodim1}, given in Appendix~\ref{app_nf}.
Therefore, we refrain from giving its details here. Let us remark that if $V$ is
of class $\cC^{r+1}$, then \eqref{scod2-3} holds with $ \order{\Norm{y}^r}$
replaced by  $\Order{\Norm{y}^{r+1}}$.

We can again apply the result to derive an easy to verify necessary condition
for a point $z$ to be a saddle. 
Let $p$ be the smallest $k$ such that $V_k$ is not identically zero.
Generically, we will have $p=3$, but other values of $p$ are quite possible, for
instance due to symmetries (see Section~\ref{ssec_zerozero} for examples).
We call \emph{discriminant}\/ the polynomial\footnote{The r\^oles of $y_1$ and
$y_2$ being interchangeable, both definitions in~\eqref{scod2-5} are
equivalent via the transformation $t\mapsto 1/t$ and multiplication by $t^p$. We
choose to make this distinction in order to avoid having to introduce \lq\lq
roots at infinity\rq\rq. Equivalently, one could work on the projective line
$\R\!P^1$.}
\begin{equation}
\label{scod2-5}
\Delta(t) = 
\begin{cases}
V_p(t,1) = V_{1\dots 11}t^p + V_{1\dots12}t^{p-1} + \dots + V_{2\dots 22}
&\quad\text{if $V_{1\dots 11}\neq0$\;,} \\
V_p(1,t) = V_{2\dots 22}t^p + V_{2\dots21}t^{p-1} + \dots + V_{1\dots 11}
&\quad\text{if $V_{1\dots 11}=0$\;.} \\
\end{cases}
\end{equation}

The following corollary as well as Table~\ref{table_codim2} provide necessary
conditions for $z$ to be a saddle, expressed in terms of $\Delta(t)$ and the
sign of $\lambda_{3}$.

\begin{table}
\begin{center}
\begin{tabular}{|l|l|l|}
\hline
& \multicolumn{1}{|c|}{$\lambda_3<0$} &  \multicolumn{1}{|c|}{$\lambda_3>0$}\\
\hline
All roots of $\Delta(t)$ real and simple & Not a saddle & {\bf Saddle} \\
No real root of $\Delta(t)$ and $\Delta(t)>0$ & {\bf Saddle} & Not a saddle
(local minimum) \\
No real root of $\Delta(t)$ and $\Delta(t)<0$ & Not a saddle & Not a saddle \\
\hline
\end{tabular}
\end{center}
\caption{Classification of stationary points with double-zero eigenvalue with
the help of the discriminant $\Delta(t)$ and the sign of $\lambda_{3}$.}
\label{table_codim2} 
\end{table}

\goodbreak

\begin{cor}\hfill
\label{cor_scodim2}
\begin{enum}
\item 	Assume $d>2$ and $\lambda_3<0$. Then the point $z$ is 
\begin{itemiz}
\item	not a saddle if the discriminant has one or several real roots, all of
them simple;
\item	a saddle if the discriminant has no real root and is positive;
\item	not a saddle if the discriminant has no real root and
is negative.
\end{itemiz}
\item 	Assume $d=2$ or $\lambda_3>0$. Then the point $z$ is 
\begin{itemiz}
\item	a saddle if the discriminant has one or several real roots, all of
them simple;
\item	not a saddle (a local minimum) if the discriminant has no real root and
is positive;
\item	not a saddle if the discriminant has no real root and is negative.
\end{itemiz}
\end{enum}
\end{cor}

\begin{proof}
In order to determine the open valley $\OV{z}$ of $z=0$, we first look for
solutions of $V(y_1,y_2,0,\dots,0)=0$ near the origin. Such solutions are
necessarily of the form  $y_1=y_2(t+R(y_2))$ for some $t$, where $R(y_2)$ goes
to zero continuously as
$y_2\to0$. Plugging in, dividing by $y_2^p$ and setting $y_2=0$, one sees that
$t$ must be a root of the discriminant. Applying the implicit function theorem
to the pair $(y_2,R)$, one finds that there is a unique $R$ if this root is
simple. Thus whenever the discriminant has real roots, all of them simple,
$V(y_1,y_2,0,\dots,0)$ changes sign in a neighbourhood of the origin, the
regions of constant sign being shaped like sectors.
If $\lambda_3<0$, we have to distinguish between three cases:
\begin{enum}
\item	The discriminant has simple real roots. In the plane
$y_{3}=\dots=y_d=0$, there are several disconnected regions in which $V$ is
negative. However these regions merge when $y_{3}$ becomes nonzero. Hence all
regions can be connected by a path leaving the plane  $y_{3}=\dots=y_d=0$, so
that the origin is not a saddle. 

\item	The discriminant has no real roots and is positive. 
Then for each fixed $(y_1,y_2,y_4,\dots)$, the set of $y_3$ such that $V$ is
negative is not \pathconnected, and consequently the set $\setsuch{y}{V(y)<0}$
cannot be \pathconnected\ either. By adding the origin, the latter set becomes
\pathconnected, showing that $0$ is indeed a saddle. 

\item	The discriminant has no real roots and is negative. Then either $d=3$
and the origin is a local maximum, or $d>3$ and the set of $(y_4,\dots,y_d)$ for
which $V$ is negative is \pathconnected\ for each fixed $(y_1,y_2,y_3)$. Thus
any two points in the open valley close to the origin can be connected (connect
both endpoints to points in the set $y_4=\dots=y_d=0$ by a path with constant
$(y_1,y_2,y_3)$, and then connect the two paths within the set
$y_4=\dots=y_d=0$). Thus the origin cannot be a saddle. 
\end{enum}
The proofs are analogous in the case $\lambda_3>0$.
\end{proof}

We note that if $\lambda_3<0$ and the degree $p$ of the discriminant is odd,
then the origin is usually not a saddle. For the sake of brevity, we do not
discuss here cases in which
the discriminant has nonsimple roots, because then the behaviour depends on
higher-order terms in the Taylor expansion and there is a large number of cases
to distinguish.  


\subsection{Singularities of higher codimension}
\label{ssec_scodimhigher}

Generalisation to nonquadratic saddles of higher codimension is now quite
obvious. If the potential $V$ is of class $\cC^{r}$ and the first $q$
eigenvalues $\lambda_1,\dots,\lambda_q$ of the Hessian $\hessian{V(z)}$ are
equal to
zero, the normal form can be written as 
\begin{equation}
\label{scodh-3}
V(y+g(y)) = \frac12\sum_{i=q+1}^d \lambda_i y_i^2
+\sum_{k=3}^r V_k(y_1,\dots,y_q)  
+ \order{\Norm{y}^{r}}\;,
\end{equation}
where again each $V_k$ is a homogeneous polynomial of degree $k$. Let $p$ denote
the smallest $k$ such that $V_{k}$ is not identically zero. Then the role of the
discriminant
is now played by its analogue 
\begin{equation}
 \label{scodh-4}
\Delta(t_1,\dots,t_{q-1}) = V_p(t_1,t_2,\dots,t_{q-1},1)\;, 
\end{equation} 
and the classification is analogous to the one in Table \ref{table_codim2}
(with $\lambda_{q+1}$ instead of $\lambda_3$). Equivalently, one can study the
sign of $V_p$ on a sphere of constant radius. 


\section{First-passage times for nonquadratic saddles}
\label{sec_capacity}


\subsection{Some potential theory}
\label{ssec_cappot}

Let $(x_t)_t$ be the solution of the stochastic differential equation
\eqref{intro1}. Given a measurable set $A\subset\R^d$, we denote by
$\tau_A=\inf\setsuch{t>0}{x_t\in A}$ the first-hitting time of $A$. For
sets\footnote{All subsets of $\R^d$ we consider from now on will be assumed to
be \emph{regular}, that is, their complement is a region with
continuously differentiable boundary.} $A,B\subset\R^d$, the quantity 
\begin{equation}
\label{cappot1}
h_{A,B}(x) = \bigprobin{x}{\tau_A < \tau_B}
\end{equation}
is known to satisfy the boundary value problem 
\begin{equation}
\begin{cases}
L h_{A,B}(x) = 0  &\qquad\text{for $x\in (A\cup B)^c$\;,} \\
\phantom{L} h_{A,B}(x) = 1    &\qquad\text{for $x\in A$\;,} \\
\phantom{L} h_{A,B}(x) = 0    &\qquad\text{for $x\in B$\;,} 
\end{cases}
\label{cappot2}
\end{equation}
where $L=\eps\Delta - \pscal{\nabla V(\cdot)}{\nabla}$ is the infinitesimal
generator of the diffusion $(x_t)_t$. By analogy with the electrical potential
created between two conductors at potentials $1$ and $0$, respectively,
$h_{A,B}$
is called the \defwd{equilibrium potential}\/ of $A$ and $B$. More generally,
one can define an equilibrium potential $h^\lambda_{A,B}$, defined by a
boundary value problem similar to~\eqref{cappot2}, but with
$\smash{Lh^\lambda_{A,B}=\lambda h^\lambda_{A,B}}$. However, we will not need
this generalisation here.

The \defwd{capacity}\/ of the sets $A$ and $B$ is again defined in analogy with
electrostatics as the total charge accumulated on one conductor of a capacitor,
for unit potential difference. The most useful expression for our purpose is
the integral, or \emph{Dirichlet form}, 
\begin{equation}
\label{cappot3}
\capacity_A(B) =
\eps \int_{(A\cup B)^c} \e^{-V(x)/\eps} \Norm{\nabla h_{A,B}(x)}^2 \,\6x
\bydef \Phi_{(A\cup B)^c}(h_{A,B})\;.
\end{equation}
We will use the fact that the equilibrium potential $h_{A,B}$ minimises the
Dirichlet form $\smash{\Phi_{(A\cup B)^c}}$, i.e., 
\begin{equation}
\label{cappot4}
\capacity_A(B) = \inf_{h\in\cH_{A,B}} \Phi_{(A\cup B)^c}(h)\;.
\end{equation}
Here $\cH_{A,B}$ is the space of twice weakly differentiable functions, whose
derivatives up to order $2$ are in $L^2$, and which satisfy the boundary
conditions in~\eqref{cappot2}.

Proposition~6.1 in~\cite{BEGK} shows (under some assumptions which can be
relaxed to suit our situation) that if $x$ is a (quadratic) local minimum of the
potential, then the expected first-hitting time of a set $B$ is given by 
\begin{equation}
\label{cappot5}
\bigexpecin{x}{\tau_B} = 
\frac{\displaystyle\int_{B^c} \e^{-V(y)/\eps}h_{\cB_\eps(x),B}(y)\,\6y}
{\capacity_{\cB_\eps(x)}(B)}\;.
\end{equation}
The numerator can be estimated by the Laplace method, using some rough {\it a
priori\/} estimates on  the equilibrium potential $h_{\cB_\eps(x),B}$ (which is
close to $1$ in a neighbourhood of $x$, and negligibly small elsewhere). 
In the generic situation where $x$ is indeed a quadratic local
minimum, and the saddle $z$ forms the gate from $x$ to $B$, it is known that 
\begin{equation}
\label{cappot6}
\int_{B^c} \e^{-V(y)/\eps}h_{B_\eps(x),B}(y)\,\6y = 
\frac{(2\pi\eps)^{d/2}}{\sqrt{\det(\hessian{V}(x))}} \e^{-V(x)/\eps}
\bigbrak{1+\Order{\eps^{1/2}\abs{\log\eps}}}
\;,
\end{equation}
cf.~\cite[Equation~(6.13)]{BEGK}. Thus, the crucial
quantity to be computed is the capacity in the denominator. In
the simplest case of a quadratic saddle $z$ whose Hessian has eigenvalues
$\lambda_1<0<\lambda_2\leqs\dots\leqs\lambda_d$, one finds 
\begin{equation}
\label{cappot7}
\capacity_{\cB_\eps(x)}(B) = \frac1{2\pi}
\sqrt{\frac{(2\pi\eps)^d\abs{\lambda_1}}{\lambda_2\dots\lambda_d}}
\e^{-V(z)/\eps} \bigbrak{1+\Order{\eps^{1/2}\abs{\log\eps}}}\;,
\end{equation}
cf.~\cite[Theorem~5.1]{BEGK}, which implies the standard Eyring--Kramers
formula~\eqref{intro3}.

In the sequel, we shall thus estimate the capacity in cases where the gate
between $A=\cB_\eps(x)$ and $B$ is a non-quadratic saddle. Roughly speaking, the
central result, which is proved in Section~\ref{sec_proofs}, states that if the
normal form of the saddle is of the type 
\begin{equation}
\label{main00}
V(y) = -u_1(y_1) + u_2(y_2,\dots,y_q) + \frac12 \sum_{j=q+1}^d \lambda_j y_j^2 
+ \Order{\Norm{y}^{r+1}}\;,
\end{equation}
where the functions $u_1$ and $u_2$ satisfy appropriate growth conditions, then
\begin{equation}
\label{main01}
\capacity_A(B) =   
\eps \, \frac
{\displaystyle\int_{\cB_{\delta_2}(0)} \e^{-u_2(y_2,\dots,y_q)/\eps}
\,\6y_2\dots\6y_q}
{\displaystyle\int_{-\delta_1}^{\delta_1} \e^{-u_1(y_1)/\eps} \,\6y_1}
\prod_{j=q+1}^d 
\sqrt{\frac{2\pi\eps}{\lambda_j}} 
\bigbrak{1+R(\eps)}\;,
\end{equation}
for certain $\delta_1=\delta_1(\eps)>0$, $\delta_2=\delta_2(\eps)>0$. The
remainder $R(\eps)$ goes to zero as $\eps \to 0$, with a speed depending on
$u_1$ and $u_2$. 
Once this result is established, the computation of capacities is reduced to the
computation of the integrals in~\eqref{main01}. 

\begin{remark}
\label{rem_cappot}
If $x$ is a nonquadratic local minimum, the integral in~\eqref{cappot6} can
also be estimated easily by standard Laplace asymptotics. 
Hence the extension of the Eyring--Kramers formula to flat local minima is
straightforward, and the real difficulty lies in the effect of flat saddles.
\end{remark}


\subsection{Transition times for codimension 1 singular saddles}
\label{ssec_cap}

We assume in this section that the potential is of class $\cC^5$ at least, as
this allows a better control of the error terms.
Consider first the case of a saddle $z$ such that the Hessian matrix $\hessian
V(z)$ has eigenvalues
$\lambda_1=0<\lambda_2\leqs\lambda_3\leqs\dots\leqs\lambda_d$. In this case, the
unstable direction at the saddle is non-quadratic while all stable directions
are quadratic. According to Corollary~\ref{cor_scodim1}, in the most generic
case the potential admits a normal form 
\begin{equation}
\label{cap01}
V(y) = -C_4 y_1^4 + \frac12 \sum_{j=2}^d \lambda_j y_j^2 
+ \Order{\Norm{y}^5}
\end{equation}
with $C_4>0$, i.e., the unstable direction is quartic. (Note that the saddle $z$
is at the origin $0$ of this coordinate system.)

We are interested in transition times between sets $A$ and $B$ for which the
gate $G(A,B)$ consists only of the saddle $z$. In other words, we assume that
any minimal path $\gamma\in\cP(A,B)$ admits $z$ as unique point of highest
altitude. This does not exclude the existence of other stationary points in
$\OV{z}$, i.e., the potential seen along the path $\gamma$ may have several
local
minima and maxima. 

\begin{theorem}
\label{thm_cap01}
Assume $z$ is a saddle whose normal form satisfies~\eqref{cap01}. Let $A$ and
$B$ belong to different \pathconnected\ components of $\OV{z}$ and assume that 
$G(A,B)=\set{z}$. Then 
\begin{equation}
\label{cap02}
\capacity_{A}(B) = 
\frac{2C_4^{1/4}}{\Gamma(1/4)}
\sqrt{\frac{(2\pi)^{d-1}}{\lambda_2\dots\lambda_d}} 
\,\eps^{d/2+1/4} \e^{-V(z)/\eps} 
\bigbrak{1+\Order{\eps^{1/4}\abs{\log\eps}^{5/4}}}\;,
\end{equation}
where $\Gamma$ denotes the Euler Gamma function.
\end{theorem}

The proof is given in Section~\ref{ssec_nonquad}. In the case of a quadratic
local minimum $x$, combining this result with Estimate~\eqref{cappot6}
immediately yields the following result on first-hitting times.

\begin{cor}
\label{cor_cap01}
Assume $z$ is a saddle whose normal form satisfies~\eqref{cap01}. 
Let\/ $O$ be one of the \pathconnected\ components of\/ $\OV{z}$, and suppose 
that the minimum of\/ $V$ in\/ $O$ is attained at a unique point $x$ and is 
quadratic. Let $B$ belong to a different \pathconnected\ component of\/
$\OV{z}$, with $G(\set{x},B)=\set{z}$. Then the
expected first-hitting time of $B$ satisfies 
\begin{equation}
\label{cap03}
\bigexpecin{x}{\tau_B} = 
\frac{\Gamma(1/4)}{2C_4^{1/4}} 
\sqrt{\frac{2\pi\lambda_2\dots\lambda_d}
{\det(\hessian{V}(x))}} \eps^{-1/4}
\e^{[V(z)-V(x)]/\eps} 
\bigbrak{1+\Order{\eps^{1/4}\abs{\log\eps}^{5/4}}}\;.
\end{equation}
\end{cor}

Note in particular that unlike in the case of a quadratic saddle, the
subexponential
asymptotics depends on $\eps$ to leading order, namely proportionally to
$\eps^{-1/4}$. 

\begin{remark} \hfill
\begin{enum}
\item	If the gate $G(A,B)$ contains several isolated saddles, the
capacity is obtained simply by adding the contributions of each individual
saddle. In other words, just as in electrostatics, for capacitors in parallel
the equivalent capacity is obtained by adding the capacities of individual
capacitors.

\item	We can extend this result to even flatter unstable directions. Assume
that the potential $V$ is of class $2p+1$ for some $p\geqs2$, and
that the normal form at the origin reads 
\begin{equation}
\label{cap04}
V(y) = -C_{2p} y_1^{2p} + \frac12 \sum_{j=2}^d \lambda_j y_j^2 
+ \Order{\Norm{y}^{2p+1}}\;,
\end{equation}
with $C_{2p}>0$.
Then a completely analogous proof shows that~\eqref{cap02} is to be replaced by 
\begin{equation}
\label{cap05}
\capacity_A(B) = 
\frac{pC_{2p}^{1/2p}}{\Gamma(1/2p)}
\sqrt{\frac{(2\pi)^{d-1}}{\lambda_2\dots\lambda_d}} \,\eps^{d/2+(p-1)/2p}
\bigbrak{1+\Order{\eps^{1/2p}\abs{\log\eps}^{(2p+1)/2p}}}\;.
\end{equation}
Consequently, if the assumptions of Corollary~\ref{cor_cap01} on the minimum $x$
and the set $B$ are satisfied, then 
\begin{equation}
\label{cap03a}
\bigexpecin{x}{\tau_B} = 
\frac{\Gamma(1/2p)}{pC_{2p}^{1/2p}} 
\sqrt{\frac{2\pi\lambda_2\dots\lambda_d}
{\det(\hessian{V}(x))}} \eps^{-(p-1)/2p}
\e^{[V(z)-V(x)]/\eps} 
\bigbrak{1+\Order{\eps^{1/2p}\abs{\log\eps}^{(2p+1)/2p}}}\;.
\end{equation}
Note that the subexponential prefactor of the expected first-hitting
time now behaves like $\eps^{-(p-1)/2p}$. As $p$ varies from $1$ to $\infty$,
i.e., as the unstable direction becomes flatter and flatter, the prefactor's
dependence on $\eps$ changes from order~$1$ to order~$1/\sqrt\eps$.
\end{enum}
\end{remark}

Consider next the case of a saddle $z$ such that the Hessian matrix $\hessian
V(z)$ has eigenvalues $\lambda_1 < \lambda_2=0 <
\lambda_3\leqs\dots\leqs\lambda_d$. In this case, all directions, whether stable
or unstable, are quadratic but for one of the stable directions. According
to Corollary~\ref{cor_scodim1}, in the most generic case the potential admits a
normal form 
\begin{equation}
\label{cap06}
V(y) = -\frac12 \abs{\lambda_1}y_1^2 + C_4 y_2^4 + \frac12 \sum_{j=3}^d
\lambda_j y_j^2 
+ \Order{\Norm{y}^5}
\end{equation}
with $C_4>0$, i.e., the non-quadratic stable direction is quartic.

\begin{theorem}
\label{thm_cap02}
Assume $z$ is a saddle whose normal form satisfies~\eqref{cap06}. Let $A$ and
$B$ belong to different \pathconnected\ components of $\OV{z}$ and assume that 
$G(A,B)=\set{z}$. Then 
\begin{equation}
\label{cap07}
\capacity_A(B) = 
\frac{\Gamma(1/4)}{2C_4^{1/4}}
\sqrt{\frac{(2\pi)^{d-3}\abs{\lambda_1}}{\lambda_3\dots\lambda_d}}
\,\eps^{d/2-1/4} \e^{-V(z)/\eps} 
\bigbrak{1+\Order{\eps^{1/4}\abs{\log\eps}^{5/4}}}\;.
\end{equation}
\end{theorem}

The proof is given in Section~\ref{ssec_nonquad}. In the case of a quadratic
local minimum $x$, combining this result with Estimate~\eqref{cappot6}
immediately yields the following result on first-hitting times.

\begin{cor}
\label{cor_cap02}
Assume $z$ is a saddle whose normal form satisfies~\eqref{cap06}. Let\/ $O$ be
one of the \pathconnected\ components of\/ $\OV{z}$, and suppose 
that the minimum of\/ $V$ in\/ $O$ is attained at a unique point $x$ and is 
quadratic. Let $B$ belong to a different \pathconnected\ component of\/
$\OV{z}$, with $G(\set{x},B)=\set{z}$. Then the expected first-hitting time of
$B$ satisfies 
\begin{equation}
\label{cap08}
\bigexpecin{x}{\tau_B} = 
\frac{2C_4^{1/4}}{\Gamma(1/4)} 
\sqrt{\frac{(2\pi)^3\lambda_3\dots\lambda_d}
{\abs{\lambda_1}\det(\hessian{V}(x))}} \eps^{1/4}
\e^{[V(z)-V(x)]/\eps} 
\bigbrak{1+\Order{\eps^{1/4}\abs{\log\eps}^{5/4}}}\;.
\end{equation}
\end{cor}

Note again the $\eps$-dependence of the prefactor, which is now proportional to
$\eps^{1/4}$ to leading order. A similar result is easily obtained in the case
of the leading term in the normal form having order $y_2^{2p}$ for some
$p\geqs2$. In particular, the prefactor of the transition time then has leading
order $\eps^{(p-1)/2p}$:
\begin{equation}
\label{cap08a}
\bigexpecin{x}{\tau_B} = 
\frac{pC_{2p}^{1/2p}}{\Gamma(1/2p)} 
\sqrt{\frac{(2\pi)^{3}\lambda_3\dots\lambda_d}
{\abs{\lambda_{1}}\det(\hessian{V}(x))}} \eps^{(p-1)/2p}
\e^{[V(z)-V(x)]/\eps} 
\bigbrak{1+\Order{\eps^{1/2p}\abs{\log\eps}^{(2p+1)/2p}}}\;.
\end{equation}
As $p$ varies from $1$ to $\infty$, i.e., as the non-quadratic stable direction
becomes flatter and flatter, the prefactor's dependence on $\eps$ changes from
order~$1$ to order~$\sqrt\eps$.


\subsection{Transition times for higher-codimension singular saddles}
\label{ssec_cap2}

We assume in this section that the potential is at least of class $\cC^{5}$.
Consider the case of a saddle $z$ such that the Hessian matrix
$\hessian V(z)$ has eigenvalues 
\begin{equation}
\label{cap00_1}
\lambda_1<0=\lambda_2=\lambda_3<\lambda_4\leqs\dots\leqs\lambda_d\;.
\end{equation}
In this case, the unstable direction is quadratic, while two of the stable
directions are non-quadratic. Proposition~\ref{prop_scodim2} shows that near the
saddle, the potential admits a normal form 
\begin{equation}
\label{cap00_2}
V(y) = -\frac12\abs{\lambda_1}y_1^2 + V_3(y_2,y_3) + V_4(y_2,y_3) + 
\frac12\sum_{j=4}^d \lambda_j y_j^2 + \Order{\Norm{y}^5}\;,
\end{equation}
with $V_3$ and $V_4$ homogeneous polynomials of degree $3$ and $4$,
respectively. If $V_3$ does not vanish identically, Corollary~\ref{cor_scodim2}
shows that $z$ is typically not a saddle. We assume thus that $V_3$ is
identically zero, and, again in view of Corollary~\ref{cor_scodim2}, that the
discriminant 
\begin{equation}
\label{cap00_3}
\Delta(t) = 
V_{2222}t^4 + V_{2223}t^3 + V_{2233}t^2 + V_{2333}t + V_{3333}
\end{equation}
has no real roots, and is positive with $V_{2222}>0$. It is convenient to
introduce polar coordinates, writing 
\begin{equation}
\label{cap00_4}
V_4(r\cos\ph,r\sin\ph) = r^4 k(\ph)\;,
\end{equation}
where we may assume that $k(\ph)$ is bounded above and below by strictly
positive constants $K_+\geqs K_-$. Then we have the following result,
which is proved in Section~\ref{ssec_nonquad}.

\begin{theorem}
\label{thm_cap001}
Assume $z$ is a saddle whose normal form satisfies~\eqref{cap00_2}, with
$V_3\equiv0$ and $V_4>0$. Suppose, the discriminant $\Delta(t)$ has no real
roots and satisfies $V_{2222}>0$. 
Let $A$ and
$B$ belong to different \pathconnected\ components of $\OV{z}$ and assume that 
$G(A,B)=\set{z}$. Then 
\begin{equation}
\label{cap00_5}
\capacity_{A}(B) = 
\frac{\sqrt{\pi}}{4} \int_0^{2\pi} \frac{\6\ph}{k(\ph)^{1/2}}
\sqrt{\frac{(2\pi)^{d-4}\abs{\lambda_1}}{\lambda_4\dots\lambda_d}} 
\,\eps^{d/2-1/2} \e^{-V(z)/\eps} 
\bigbrak{1+\Order{\eps^{1/4}\abs{\log\eps}^{5/4}}}\;.
\end{equation}
\end{theorem}

\begin{cor}
\label{cor_cap001}
Assume $z$ is a saddle whose normal form satisfies~\eqref{cap00_2}. Let\/ $O$ be
one of the \pathconnected\ components of\/ $\OV{z}$, and assume that the minimum
of\/ $V$ in\/ $O$ is reached at a unique point $x$, which is quadratic.  Let $B$
belong to a different \pathconnected\ component of\/ $\OV{z}$, with
$G(\set{x},B)=\set{z}$. Then the expected first-hitting time of $B$ satisfies 
\begin{equation}
\label{cap00_6}
\bigexpecin{x}{\tau_B} = 
\frac{4}{\sqrt{\pi}\displaystyle\int_0^{2\pi} \frac{\6\ph}{k(\ph)^{1/2}}}
\sqrt{\frac{(2\pi)^4\lambda_4\dots\lambda_d}
{\abs{\lambda_1}\det(\hessian{V}(x))}} \eps^{1/2}
\e^{[V(z)-V(x)]/\eps} 
\bigbrak{1+\Order{\eps^{1/4}\abs{\log\eps}^{5/4}}}\;.
\end{equation}
\end{cor}

The prefactor is now proportional to $\eps^{1/2}$ instead of being proportional
to $\eps^{1/4}$, which is explained by the presence of two vanishing
eigenvalues. 

\begin{remark}
\label{rem_cap01}
This result admits two straightforward generalisations to less generic
situations:
\begin{enum}
\item	If the potential is of class $\cC^{2p+1}$ and the first nonvanishing
coefficient of the normal form has even degree $2p$, $p\geqs2$, then a
completely
analogous proof shows that 
\begin{multline}
\label{cap00_7}
\capacity_{A}(B) = \frac1{2p} \Gamma\biggpar{\frac1p}
\int_0^{2\pi} \frac{\6\ph}{k(\ph)^{1/p}}
\sqrt{\frac{(2\pi)^{d-4}\abs{\lambda_1}}{\lambda_4\dots\lambda_d}} 
\,\eps^{d/2-(p-1)/p} \e^{-V(z)/\eps} \\
\times
\bigbrak{1+\Order{\eps^{1/2p}\abs{\log\eps}^{(2p+1)/2p}}}\;.
\end{multline}
Consequently, if the assumptions of Corollary~\ref{cor_cap001} on the minimum
$x$ and the set $B$ are satisfied, then 
\begin{multline}
\label{cap00_7b}
\bigexpecin{x}{\tau_B} = 
\frac{2p}{\displaystyle\Gamma\biggpar{\frac1p}}
\frac{1}{\displaystyle\int_0^{2\pi} \frac{\6\ph}{k(\ph)^{1/p}}}
\sqrt{\frac{(2\pi)^4\lambda_4\dots\lambda_d}
{\abs{\lambda_1}\det(\hessian{V}(x))}} \eps^{(p-1)/p}
\e^{[V(z)-V(x)]/\eps} \\
\times
\bigbrak{1+\Order{\eps^{1/2p}\abs{\log\eps}^{(2p+1)/2p}}}\;.
\end{multline}

\item	If all eigenvalues from $\lambda_2$ to $\lambda_q$ are equal to zero,
for some $q\geqs4$ and the first nonvanishing coefficient of the normal form has
even degree $2p$, $p\geqs2$, then the prefactor of the capacity has order
$\eps^{d/2-(q-1)(p-1)/2p}$, and involves a $(q-2)$-dimensional integral over the
angular part of the leading term in the normal form. 
\end{enum}
\end{remark}

The other important codimension-two singularity occurs for a saddle $z$ such
that the Hessian matrix $\hessian V(z)$ has eigenvalues 
\begin{equation}
\label{cap00_10}
0=\lambda_1=\lambda_2<\lambda_3\leqs\dots\leqs\lambda_d\;.
\end{equation}
In this case, Corollary~\ref{cor_scodim2} states that $z$ is a saddle when the
discriminant of the normal form has one or more real roots, all of them simple.
As a consequence, there can be more than two valleys meeting at the saddle. This
actually induces a serious difficulty for the estimation of the capacity. The
reason is that for this estimation, one needs an {\it a priori\/} bound on the
equilibrium potential $h_{A,B}$ in the valleys, some distance away from the
saddle. In cases with only two valleys, $h_{A,B}$ is very close to $1$ in the
valley containing $A$, and very close to $0$ in the valley containing $B$. When
there are additional valleys, however, one would first have to obtain an {\it a
priori\/} estimate on the value of $h_{A,B}$ in these valleys, which is not at
all straightforward, except perhaps in situations involving symmetries. We will
not discuss this case here. 

Arguably, the study of singular saddles satisfying~\eqref{cap00_10} is less
important, because they are less stable against perturbations of the potential.
Namely, the flatness of the potential around the saddle implies that like for
the saddle--node bifurcation, there exist perturbations of the potential
that do no longer admit a saddle close by. As a consequence, there is no
potential barrier creating metastability for these perturbations.



\section{Bifurcations}
\label{sec_bif}

While the results in the previous section describe the situation for
nonquadratic saddles, that is, at a bifurcation point, they do not incorporate 
the transition from quadratic to nonquadratic saddles. In order to complete the
picture, we now give a description of the metastable behaviour in a full
neighbourhood of a bifurcation point of a parameter-dependent potential. We will
always assume that $V$ is of class $\cC^{5}$.

We shall discuss a few typical examples of bifurcations, which we will
illustrate on
the model potential
\begin{equation}
 \label{bif01}
V_\gamma(x) = \sum_{i=1}^N U(x_i) + 
\frac{\gamma}{4}\sum_{i=1}^N (x_i-x_{i+1})^2\;, 
\end{equation} 
introduced in~\cite{BFG06a}. Here $x_i$ denotes
the position of a particle attached to site $i$ of the lattice $\Z/N\Z$, 
$U(x_i)=\frac14x_i^4-\frac12x_i^2$ is a local double-well potential acting on
that particle, and the second sum describes a harmonic ferromagnetic
interaction between neighbouring particles (with the identification
$x_{N+1}=x_1$). Indeed, for $N=2$ and $\gamma=1/2$, the origin is a nonquadratic
saddle of codimension $1$ of the potential~\eqref{bif01}, while for all
$N\geqs3$, the origin is a nonquadratic saddle of codimension $2$ when 
$\gamma=(2\sin^2(\pi/N))^{-1}$.


\begin{figure}
\vspace{7mm}
\centerline{\includegraphics*[clip=true,width=140mm]{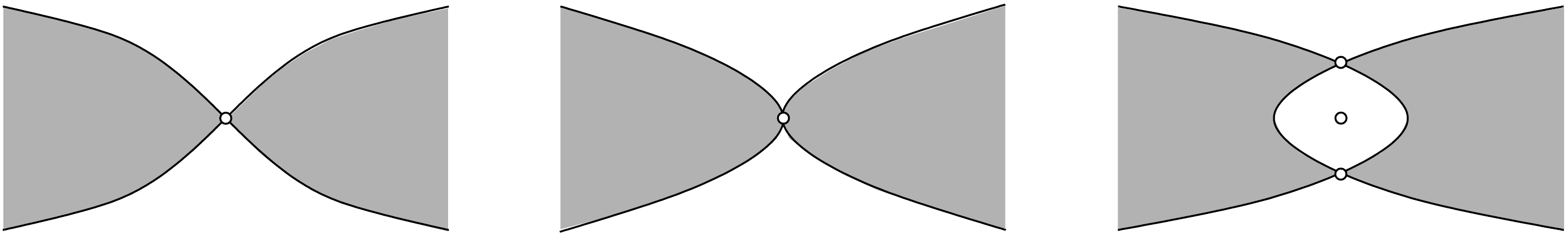}}
 \figtext{ 
	\writefig	0.3	2.8	{\bf (a)}
	\writefig	5.3	2.8	{\bf (b)}
	\writefig	10.3	2.8	{\bf (c)}
	\writefig	2.3	1.7	$z$
	\writefig	0.7	1.4	$O_-$
	\writefig	3.7	1.4	$O_+$
	\writefig	7.3	1.8	$z$
	\writefig	5.7	1.4	$O_-$
	\writefig	8.7	1.4	$O_+$
	\writefig	12.45	1.45	$z$
	\writefig	12.25	0.7	$z_-$
	\writefig	12.25	2.2	$z_+$
	\writefig	10.7	1.4	$O_-$
	\writefig	13.7	1.4	$O_+$
}
\caption[]{Saddles and open valleys of the normal-form
potential~\eqref{pitch01}, in a two-dimensional case, {\bf (a)} for
$\lambda_2>0$, {\bf (b)} for $\lambda_2=0$ and {\bf (c)} for $\lambda_2<0$. The
system undergoes a transversal pitchfork bifurcation at $\lambda_2=0$.} 
\label{fig_transpitch}
\end{figure}

\subsection{Transversal symmetric pitchfork bifurcation}
\label{ssec_pitch}

Let us assume that the potential $V$ depends continuously on a parameter
$\gamma$, and that for $\gamma=\gamma^\star$, $z=0$ is a nonquadratic saddle of
$V$, with normal form~\eqref{cap06}. A symmetric pitchfork bifurcation occurs
when for $\gamma$ near $\gamma^\star$, the normal form has the expression
\begin{equation}
\label{pitch01}
V(y) = \frac12 \lambda_1(\gamma)y_1^2 + \frac12
\lambda_2(\gamma)y_2^2 + C_4(\gamma) y_2^4 + \frac12 \sum_{j=3}^d
\lambda_j(\gamma) y_j^2 
+ \Order{\Norm{y}^5}\;,
\end{equation}
where $\lambda_2(\gamma^\star)=0$, while $\lambda_1(\gamma^\star)<0$,
$C_4(\gamma^\star)>0$, and similarly for the other quantities. We assume
here that $V$ is even in $y_2$, which is the most common situation in which
pitchfork bifurcations are observed. For simplicity, we shall usually refrain
from indicating the $\gamma$-dependence of the eigenvalues in the sequel. All
quantities except $\lambda_2$ are assumed to be bounded away from zero as
$\gamma$ varies.

When $\lambda_2>0$, $z=0$ is a quadratic saddle. When $\lambda_2<0$, $z=0$ is no
longer a saddle (the origin then having a two-dimensional unstable manifold),
but there exist two saddles $z_\pm$ with coordinates
\begin{equation}
\label{pitch02}
z_\pm = 
\bigpar{0,\pm\sqrt{\abs{\lambda_2}/4C_4}+\Order{\lambda_2}
,0,\dots,0}+\Order{\lambda_2^{2}}
\end{equation}
(\figref{fig_transpitch}).
Let us denote the eigenvalues of $\hessian{V}(z_\pm)$ by $\mu_1,\dots,\mu_d$. 
In fact, for $\lambda_2<0$ we have
\begin{align}
\nonumber
\mu_2 &= -2\lambda_2 + \Order{\abs{\lambda_2}^{3/2}}\;, \\
\mu_j &= \lambda_j+\Order{\abs{\lambda_2}^{3/2}}
&&\text{for $j\in\set{1,3,\dots,d}$\;.}
\label{pitch03}
\end{align}
Finally, the value of the potential on the saddles $z_\pm$ satisfies 
\begin{equation}
\label{pitch04}
V(z_+) = V(z_-) = V(z) - \frac{\lambda_2^2}{16C_4} +
\Order{\abs{\lambda_2}^{5/2}}\;.
\end{equation}

\begin{figure}
\centerline{\includegraphics*[clip=true,width=140mm]{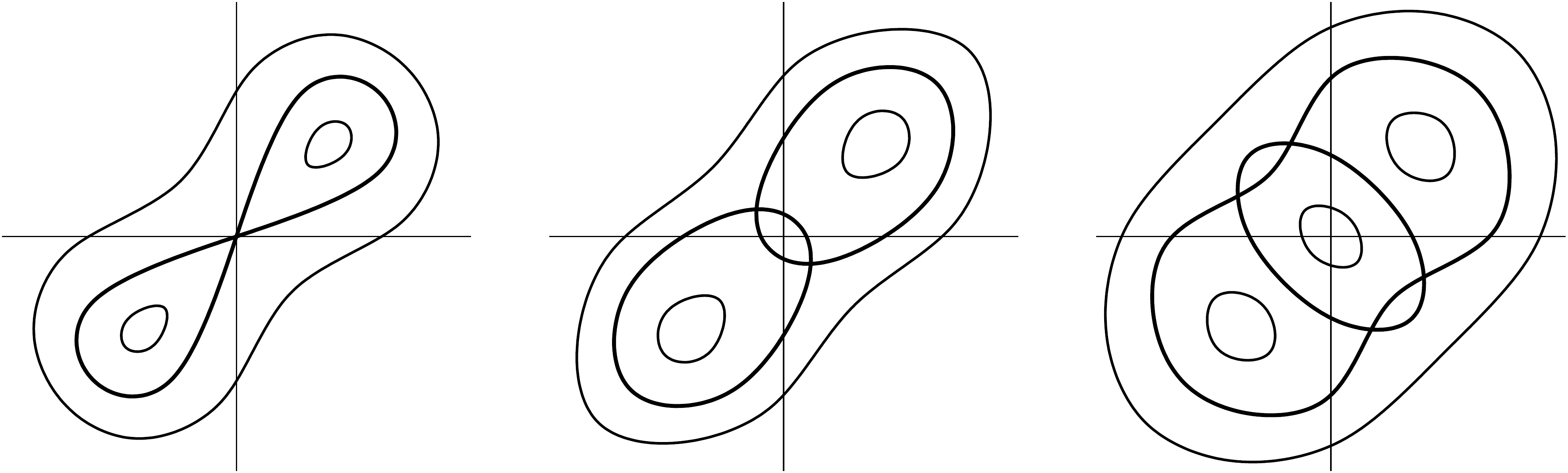}}
 \figtext{ 
	\writefig	0.3	4.5	{\bf (a)}
	\writefig	5.2	4.5	{\bf (b)}
	\writefig	10.1	4.5	{\bf (c)}
 }
\caption[]{Level lines of the potential~\eqref{pitchex1}, {\bf (a)} for
$\gamma>1/2$, {\bf (b)} for $1/2>\gamma>1/3$ and {\bf (c)} for $1/3>\gamma>0$.
A transversal pitchfork bifurcation occurs at $\gamma=1/2$, and a longitudinal
pitchfork bifurcation occurs at $\gamma=1/3$.}
\label{fig_N2pot}
\end{figure}

\begin{example}
\label{ex_pitch1}
For $N=2$ particles, the potential~\eqref{bif01} reads 
\begin{equation}
\label{pitchex1}
V(x_1,x_2) = U(x_1) + U(x_2) + \frac{\gamma}{2}(x_1-x_2)^2\;.
\end{equation}
Performing a rotation by $\pi/4$ yields the equivalent potential  
\begin{equation}
\label{pitchex2}
\widehat V(y_1,y_2) = 
-\frac12 y_1^2 - \frac{1-2\gamma}{2}y_2^2 + \frac18 (y_1^4+6y_1^2y_2^2+y_2^4)\;,
\end{equation}
which immediately shows that the origin $(0,0)$ is a stationary point with
$\lambda_1(\gamma)=-1$ and $\lambda_2(\gamma)=-(1-2\gamma)$. For
$\gamma>\gamma^\star=1/2$, the origin is thus a quadratic saddle, at
\lq\lq altitude\rq\rq~$0$. It serves as a gate between the local minima located
at
$y=(\pm1/\sqrt2,0)$. As $\gamma$ decreases below $\gamma^\star$, two new saddles
appear at \mbox{$y=(0,\pm\sqrt{2(1-2\gamma)}$} (cf.~\figref{fig_N2pot} which
shows the potential's level lines in the original variables $(x_{1},x_{2})$).
They have a
positive eigenvalue
$\mu_2(\gamma)=2(2\gamma-1)$, and the \lq\lq altitude\rq\rq\
$-\frac12(1-2\gamma)^2$. There is
thus a pitchfork bifurcation at $\gamma=1/2$. Note that another pitchfork
bifurcation, affecting the new saddles, occurs at $\gamma=1/3$.
\end{example}

Our main result is the following sharp estimate of the capacity.

\begin{figure}
\centerline{\includegraphics*[clip=true,width=75mm]{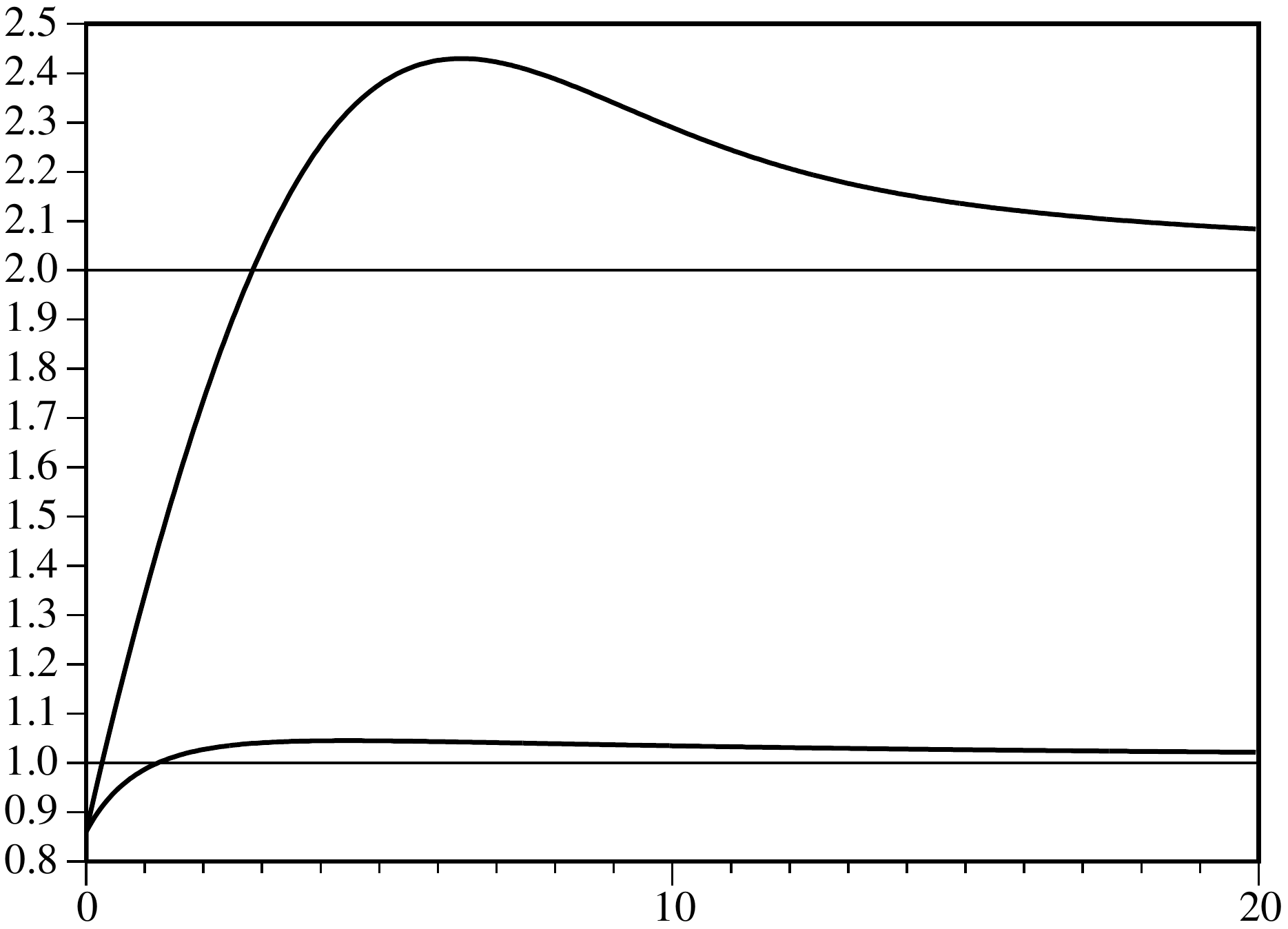}
\includegraphics*[clip=true,width=75mm]{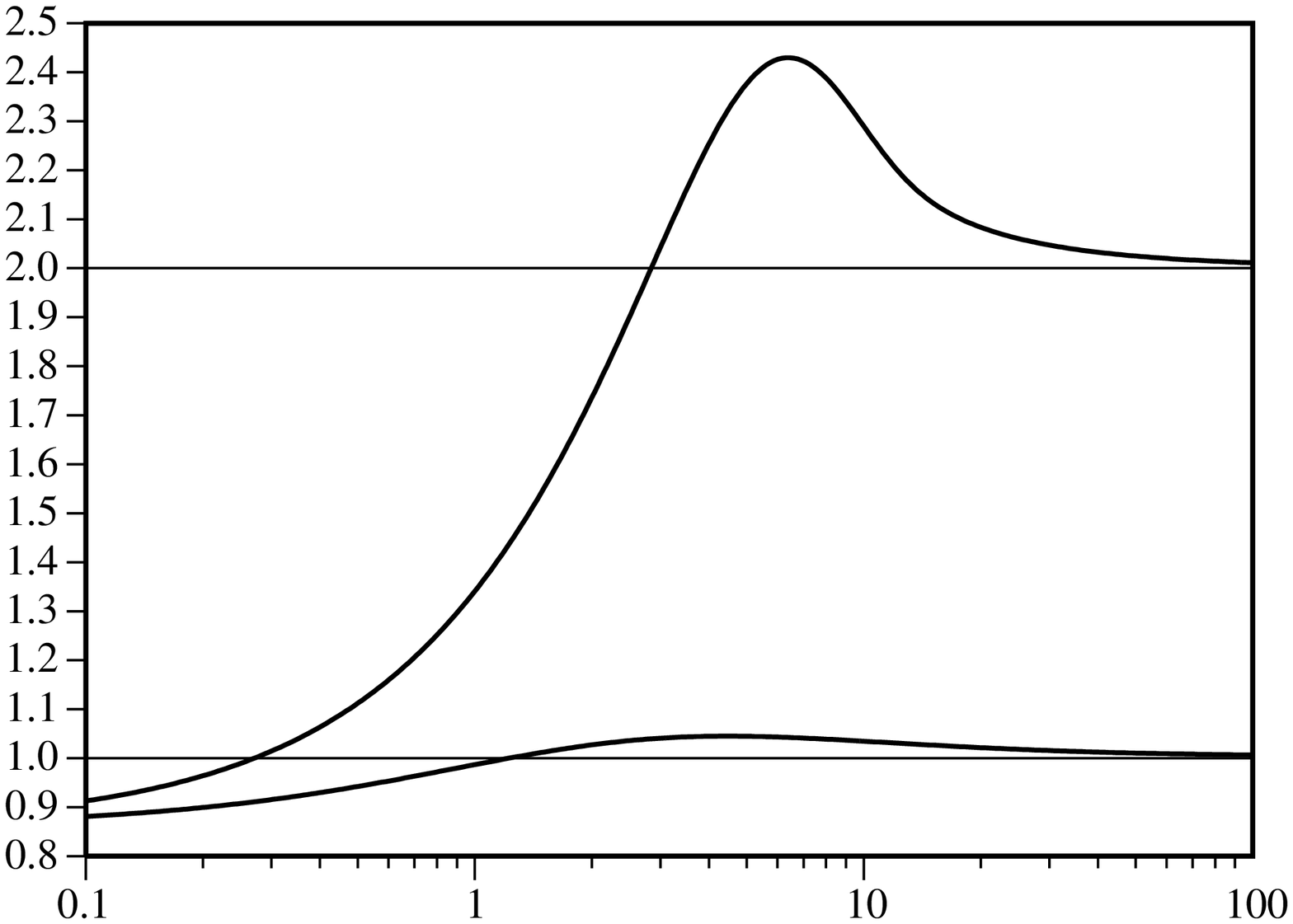}}
 \figtext{ 
	\writefig	3.0	1.8	$\Psi_+(\alpha)$
	\writefig	4.4	5.0	$\Psi_-(\alpha)$
	\writefig	11.0	1.8	$\Psi_+(\alpha)$
	\writefig	10.3	4.8	$\Psi_-(\alpha)$
 }
\caption[]{The functions $\Psi_\pm(\alpha)$, shown on a linear and on a
logarithmic scale.}
\label{fig_Psi}
\end{figure}

\begin{theorem}
\label{thm_pitch1}
Assume the normal form at $z=0$ satisfies~\eqref{pitch01}. Let $A$ and
$B$ belong to different \pathconnected\ components of $\OV{z}$ (respectively of
$\OV{z_\pm}$ if $\lambda_2<0$). Assume further that $G(A,B)=\set{z}$
(resp.\ $G(A,B)=\set{z_-,z_+}$ if $\lambda_2<0$). Then for $\lambda_2>0$, 
\begin{equation}
\label{pitch05}
\capacity_A(B) = \sqrt{\frac{(2\pi)^{d-2}\abs{\lambda_1}}
{\brak{\lambda_2+(2\eps C_4)^{1/2}}\lambda_3\dots\lambda_d}}\,
\Psi_+\biggpar{\frac{\lambda_2}{(2\eps C_4)^{1/2}}} \eps^{d/2} 
\e^{-V(z)/\eps}
\bigbrak{1+R_+(\eps,\lambda_2)}\;,
\end{equation}
while for $\lambda_2<0$, 
\begin{equation}
\label{pitch06}
\capacity_A(B) = \sqrt{\frac{(2\pi)^{d-2}\abs{\mu_1}}
{\brak{\mu_2+(2\eps C_4)^{1/2}}\mu_3\dots\mu_d}}\,
\Psi_-\biggpar{\frac{\mu_2}{(2\eps C_4)^{1/2}}} \eps^{d/2} 
\e^{-V(z_\pm)/\eps}
\bigbrak{1+R_-(\eps,\mu_2)}\;. 
\end{equation}
The functions $\Psi_+$ and $\Psi_-$ are bounded above and below uniformly on
$\R_+$. They admit the explicit expressions 
\begin{align}
\nonumber
\Psi_+(\alpha) &= \sqrt{\frac{\alpha(1+\alpha)}{8\pi}} \e^{\alpha^2/16}
K_{1/4}\biggpar{\frac{\alpha^2}{16}}\;,\\
\Psi_-(\alpha) &= \sqrt{\frac{\pi\alpha(1+\alpha)}{32}} \e^{-\alpha^2/64}
\biggbrak{I_{-1/4}\biggpar{\frac{\alpha^2}{64}}
+I_{1/4}\biggpar{\frac{\alpha^2}{64}}}\;,
\label{pitch07}
\end{align}
where $K_{1/4}$ and $I_{\pm1/4}$ denote the modified Bessel functions of the
second
and first kind, respectively. In particular,
\begin{equation}
\label{pitch08}
\lim_{\alpha\to+\infty} \Psi_+(\alpha) = 1\;,
\qquad
\lim_{\alpha\to+\infty} \Psi_-(\alpha) = 2\;,
\end{equation}
and 
\begin{equation}
\label{pitch09}
\lim_{\alpha\to0} \Psi_+(\alpha) = \lim_{\alpha\to0} \Psi_-(\alpha) 
= \frac{\Gamma(1/4)}{2^{5/4}\sqrt{\pi}} \simeq 0.8600\;.
\end{equation}
Finally, the error terms satisfy 
\begin{equation}
\label{pitch10}
\bigabs{R_\pm(\eps,\lambda)} \leqs C 
\biggbrak{
\frac{\eps\abs{\log\eps}^3}
{\max\bigset{\abs{\lambda}, (\eps\abs{\log\eps})^{1/2}}}}^{1/2} 
\;.
\end{equation}
\end{theorem}

The functions $\Psi_\pm(\alpha)$ are universal, in the sense that they will be
the same for all symmetric pitchfork bifurcations, regardless of the details of
the system. They are shown in \figref{fig_Psi}. Note in
particular that they are not monotonous, but both admit a maximum. 

\begin{cor}
\label{cor_pitch1}
Assume the normal form at $z=0$ satisfies~\eqref{pitch01}. 
\begin{itemiz}
\item	If $\lambda_2>0$, assume that the minimum of\/ $V$ in one of the
\pathconnected\ components of\/ $\OV{z}$ is reached at a unique point
$x$, which is quadratic. Let $B$ belong to a different \pathconnected\ component
of\/ $\OV{z}$, with $G(\set{x},B)=\set{z}$. Then the expected first-hitting time
of $B$ satisfies 
\begin{equation}
\label{pitch11}
\bigexpecin{x}{\tau_B} = 
2\pi \sqrt{\frac{\brak{\lambda_2+(2\eps C_4)^{1/2}}\lambda_3\dots\lambda_d}
{\abs{\lambda_1}\det(\hessian{V}(x))}}
\frac{\e^{[V(z)-V(x)]/\eps}}{\Psi_+\bigpar{\lambda_2/(2\eps C_4)^{1/2}}}
\bigbrak{1+R_+(\eps,\lambda_2)}\;.
\end{equation}
 
\item	If $\lambda_2<0$, assume that the minimum of\/ $V$ in one of the
\pathconnected\ components of\/ $\OV{z_+}=\OV{z_-}$ is reached at a unique point
$x$,
which is quadratic. Let $B$ belong to a different \pathconnected\ component of\/
$\OV{z_\pm}$, with $G(\set{x},B)=\set{z_+,z_-}$. Then the expected first-hitting
time of $B$ satisfies 
\begin{equation}
\label{pitch12}
\bigexpecin{x}{\tau_B} = 
2\pi \sqrt{\frac{\brak{\mu_2+(2\eps C_4)^{1/2}}\mu_3\dots\mu_d}
{\abs{\mu_1}\det(\hessian{V}(x))}}
\frac{\e^{[V(z_\pm)-V(x)]/\eps}}{\Psi_-\bigpar{\mu_2/(2\eps C_4)^{1/2}}}
\bigbrak{1+R_-(\eps,\mu_2)}\;.
\end{equation}
\end{itemiz}
\end{cor}

\begin{figure}
\centerline{\includegraphics*[clip=true,width=75mm]{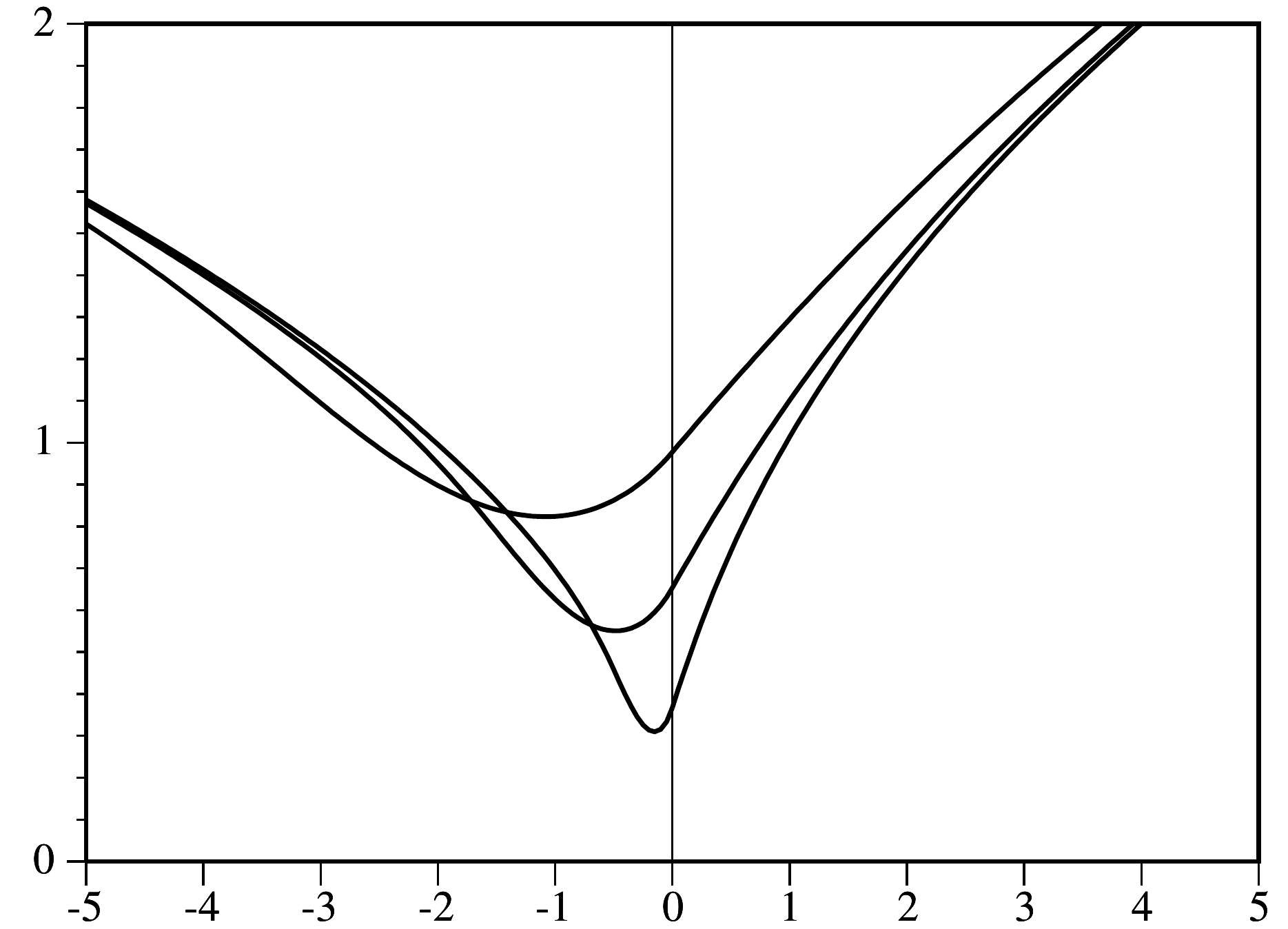}}
 \figtext{ 
	\writefig	10.5	1.0	$\lambda_2$
	\writefig	7.8	4.8	$\eps=0.5$
	\writefig	5.6	2.3	$\eps=0.1$
	\writefig	7.8	2.0	$\eps=0.01$
 }
\caption[]{The prefactor of the expected transition time near a pitchfork
bifurcation, as a function of the bifurcation parameter $\lambda_2$, shown for
three different values of $\eps$. (To be precise, we show the function
$\lambda_2\mapsto\sqrt{\lambda_2+\eps^{1/2}}/\Psi_+(\lambda_2/\eps^{1/2})$ for
$\lambda_2>0$ and the function
$\lambda_2\mapsto\sqrt{-2\lambda_2+\eps^{1/2}}/\Psi_-(-2\lambda_2/\eps^{1/2})$
for $\lambda_2<0$.)}
\label{fig_Psi2}
\end{figure}

When $\lambda_2$ is bounded away from zero, the expression~\eqref{pitch11}
reduces to the usual Eyring--Kramers formula~\eqref{intro3}. When
$\lambda_2\to0$, it converges to the limiting expression~\eqref{cap08}. The
function $\Psi_+$ controls the crossover between the two regimes, which takes
place when $\lambda_2$ is of order $\eps^{1/2}$. In fact, when
$\lambda_2\ll\eps^{1/2}$, there is a saturation effect, in the sense that the
system behaves as if the curvature of the potential were bounded below by
$(2\eps C_4)^{1/2}$. Similar remarks apply to the expression~\eqref{pitch12},
the only difference being a factor $1/2$ in the prefactor when $\mu_2$ is
bounded away from $0$ (cf.~\eqref{pitch08}), which is due to the fact that the
gate between $x$ and $B$ then contains two saddles.

The $\lambda_2$-dependence of the prefactor is shown in~\figref{fig_Psi2}. It
results from the combined effect of the term under the square root and the
factors $\Psi_\pm$. Note in particular that the minimal value of the
prefactor is located at a negative value of $\lambda_2$, which can be shown to
be of order $\eps^{1/2}$.


\subsection{Longitudinal symmetric pitchfork bifurcation}
\label{ssec_lpitch}

Consider now the case where for $\gamma=\gamma^\star$, $z=0$ is a nonquadratic
saddle of $V$, with normal form~\eqref{cap01}. Then a slightly different
variant of symmetric pitchfork bifurcation occurs, when for $\gamma$ near
$\gamma^\star$, the normal form has the expression
\begin{equation}
\label{lpitch01}
V(y) = \frac12 \lambda_1(\gamma)y_1^2 - C_4(\gamma) y_1^4 + \frac12 \sum_{j=2}^d
\lambda_j(\gamma) y_j^2 
+ \Order{\Norm{y}^5}\;,
\end{equation}
where $\lambda_1(\gamma^\star)=0$, while 
$C_4(\gamma^\star)>0$ and $\lambda_j(\gamma^\star)>0$ for all $j\geqs2$. We
assume here that $V$ is even in $y_1$, which is the most common situation in
which pitchfork bifurcations are observed. As before, we no longer indicate the
$\gamma$-dependence of eigenvalues, and all quantities except $\lambda_1$ are
assumed to be bounded away from zero.

\begin{figure}
\vspace{7mm}
\centerline{\includegraphics*[clip=true,width=140mm]{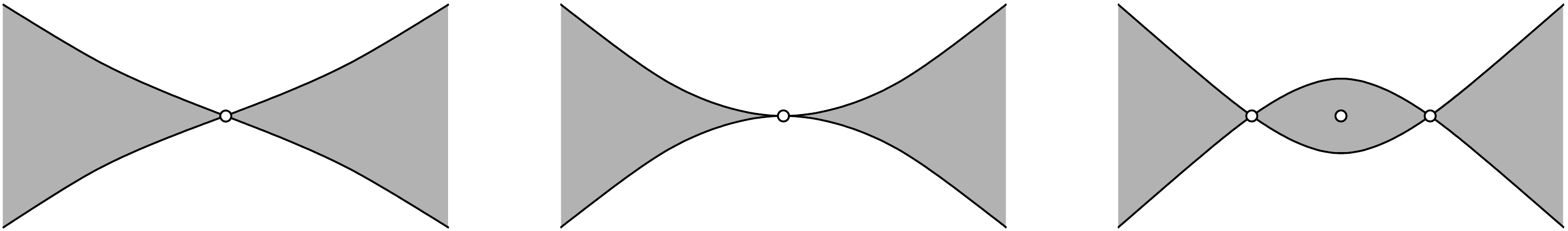}}
 \figtext{ 
	\writefig	0.3	2.8	{\bf (a)}
	\writefig	5.3	2.8	{\bf (b)}
	\writefig	10.3	2.8	{\bf (c)}
	\writefig	2.3	1.7	$z$
	\writefig	0.7	1.4	$O_-$
	\writefig	3.7	1.4	$O_+$
	\writefig	7.3	1.7	$z$
	\writefig	5.7	1.4	$O_-$
	\writefig	8.7	1.4	$O_+$
	\writefig	11.35	1.75	$z_-$
	\writefig	12.95	1.75	$z_+$
	\writefig	10.7	1.4	$O_-$
	\writefig	12.45	1.4	$O_0$
	\writefig	13.7	1.4	$O_+$
}
\caption[]{Saddles and open valleys of the normal-form
potential~\eqref{lpitch01}, {\bf (a)} for $\lambda_1<0$, {\bf (b)} for
$\lambda_1=0$ and {\bf (c)} for $\lambda_1>0$. The system undergoes a
longitudinal pitchfork bifurcation at $\lambda_1=0$.}
\label{fig_longpitch}
\end{figure}

When $\lambda_1<0$, $z=0$ is a quadratic saddle. When $\lambda_1>0$, $z=0$ is a
local minimum, but there exist two saddles $z_\pm$ with coordinates
\begin{equation}
\label{lpitch02}
z_\pm = 
\bigpar{\pm\sqrt{\lambda_1/4C_4}+\Order{\lambda_1},0,\dots,0}
+\Order{\lambda_1^{2}}
\;.
\end{equation}
The open valleys of $z_+$ and $z_-$ share a \pathconnected\ component which we
denote by $O_0$, while we denote their other components by $O_\pm$,
cf.~\figref{fig_longpitch}.  
Let us denote the eigenvalues of $\hessian{V}(z_\pm)$ by $\mu_1,\dots,\mu_d$. 
For $\lambda_1>0$ we have
\begin{align}
\nonumber
\mu_1 &= -2\lambda_1 + \Order{\lambda_1^{3/2}}\;, \\
\mu_j &= \lambda_j+\Order{\lambda_1^{3/2}}
&&\text{for $j\in\set{2,\dots,d}$\;.}
\label{lpitch03}
\end{align}
Finally, the value of the potential on the saddles $z_\pm$ satisfies 
\begin{equation}
\label{lpitch04}
V(z_+) = V(z_-) = V(z) + \frac{\lambda_1^2}{16C_4} +
\Order{\lambda_1^{5/2}}\;.
\end{equation}

Such a bifurcation occurs, for instance, in Example~\ref{ex_pitch1} for
$\gamma=1/3$. Then both saddles $z_-$ and $z_+$ have to be crossed on any
minimal path between the global minima, as in~\figref{fig_gates}c. 

We can now state a sharp estimate of the capacity in this situation, which is
proved in Section~\ref{ssec_bifs}.

\begin{theorem}
\label{thm_lpitch1}
Assume the normal form at $z=0$ satisfies~\eqref{lpitch01}. Let $A$ and
$B$ belong to the pathconnected components $O_-$ and $O_+$ respectively. Assume
further that $G(A,B)=\set{z}$ (resp.\ that $\set{z_-}$ and $\set{z_+}$ both
form a gate between $A$ and $B$ if $\lambda_1>0$). Then for $\lambda_1<0$, 
\begin{equation}
\label{lpitch05}
\capacity_A(B) =
\sqrt{\frac{(2\pi)^{d-2}\brak{\abs{\lambda_1}+(2\eps C_4)^{1/2}}}
{\lambda_2\dots\lambda_d}}\,
\frac{\eps^{d/2}} 
{\Psi_+\bigpar{\abs{\lambda_1}/(2\eps C_4)^{1/2}}} 
\e^{-V(z)/\eps}
\bigbrak{1+R_+(\eps,\abs{\lambda_1})}\;,
\end{equation}
while for $\lambda_1>0$, 
\begin{equation}
\label{lpitch06}
\capacity_A(B) =
\sqrt{\frac{(2\pi)^{d-2}\brak{\abs{\mu_1}+(2\eps C_4)^{1/2}}}
{\mu_2\dots\mu_d}}\,
\frac{\eps^{d/2}} 
{\Psi_-\bigpar{\abs{\mu_1}/(2\eps C_4)^{1/2}}}
\e^{-V(z_\pm)/\eps}
\bigbrak{1+R_-(\eps,\abs{\mu_1})}\;. 
\end{equation}
The functions $\Psi_+$ and $\Psi_-$ and the remainders $R_\pm$ are the same as
in~\eqref{pitch07} and~\eqref{pitch10}.
\end{theorem}

When $\lambda_1\gg\eps^{1/2}$, the function $\Psi_-$ in~\eqref{lpitch06} is
close to $2$, so that the total capacity equals half the capacity associated
with each saddle $z_-$ and $z_+$. As in electrostatics, when two capacitors are
set up in series, the inverse of their equivalent capacity is thus equal to the
sum of the inverses of the individual capacities. 

\begin{figure}
\centerline{\includegraphics*[clip=true,width=75mm]{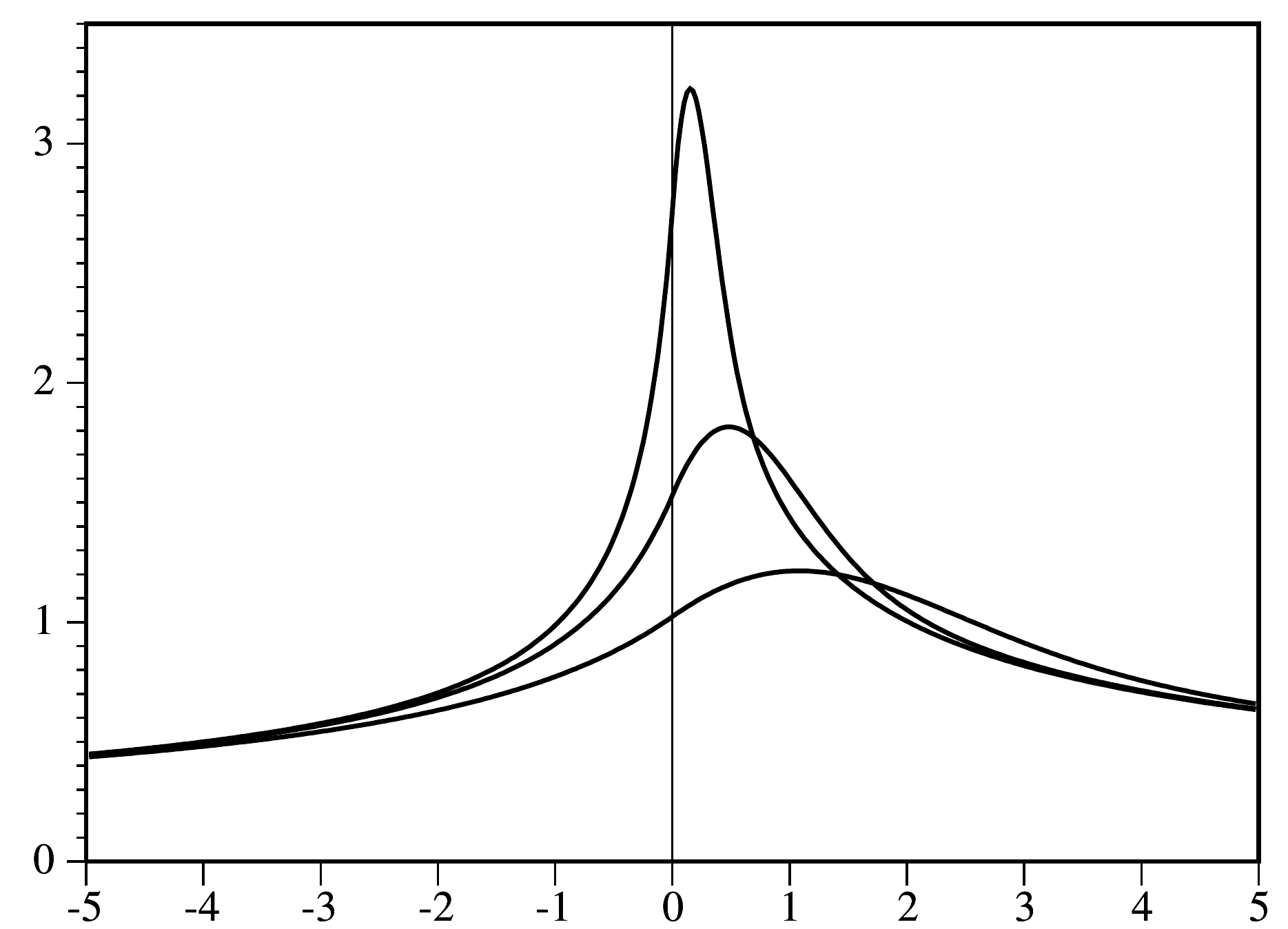}}
 \figtext{ 
	\writefig	10.5	1.0	$\lambda_1$
	\writefig	7.9	4.8	$\eps=0.01$
	\writefig	8.5	2.9	$\eps=0.1$
	\writefig	6.3	1.5	$\eps=0.5$
 }
\caption[]{The prefactor of the expected transition time near a
longitudinal pitchfork
bifurcation, as a function of the bifurcation parameter $\lambda_1$, shown for
three different values of $\eps$. (To be precise, we show the function
$\lambda_1\mapsto\Psi_+(\abs{\lambda_1}/\eps^{1/2})/\sqrt{\abs{\lambda_1}+\eps^{
1/2 } } $ for $\lambda_1<0$ and the function
$\lambda_1\mapsto\Psi_-(2\lambda_1/\eps^{1/2})/\sqrt{2\lambda_1+\eps^{1/2}}$
for $\lambda_1>0$.)}
\label{fig_Psi2long}
\end{figure}

\begin{cor}
\label{cor_lpitch1}
Assume the normal form at $z=0$ satisfies~\eqref{lpitch01}. 
\begin{itemiz}
\item	If $\lambda_1<0$, assume that the minimum of\/ $V$ in one of the
\pathconnected\ components of\/ $\OV{z}$ is reached at a unique point
$x$, which is quadratic. Let $B$ belong to a different \pathconnected\ component
of\/ $\OV{z}$, with $G(\set{x},B)=\set{z}$. Then the expected first-hitting time
of $B$ satisfies 
\begin{multline}
\label{lpitch11}
\bigexpecin{x}{\tau_B} = 
2\pi \sqrt{\frac{\lambda_2\dots\lambda_d}
{\brak{\abs{\lambda_1}+(2\eps C_4)^{1/2}}\det(\hessian{V}(x))}}
\Psi_+\biggpar{\frac{\abs{\lambda_1}}{(2\eps C_4)^{1/2}}}
\e^{[V(z)-V(x)]/\eps} \\
\times
\bigbrak{1+R_+(\eps,\abs{\lambda_1})}\;.
\end{multline}
 
\item	If $\lambda_1>0$, assume that the minimum of\/ $V$ in the
\pathconnected\ component $O_-$ of\/ $\OV{z_-}$ is reached at a unique point
$x$, which is quadratic. Let $B$ belong to the \pathconnected\ component $O_+$
of\/ $\OV{z_+}$, with $G(\set{x},B)=\set{z_+}$ or $\set{z_-}$. Assume
finally that $V(x)$ is strictly smaller than the minimum of $V$ in the common
component $O_0$ of $\OV{z_-}$ and $\OV{z_+}$. Then the
expected first-hitting time of $B$ satisfies 
\begin{multline}
\label{lpitch12}
\bigexpecin{x}{\tau_B} = 
2\pi \sqrt{\frac{\mu_2\dots\mu_d}
{\brak{\abs{\mu_1}+(2\eps C_4)^{1/2}}\det(\hessian{V}(x))}}
\Psi_-\biggpar{\frac{\abs{\mu_1}}{(2\eps C_4)^{1/2}}}
\e^{[V(z_\pm)-V(x)]/\eps} \\
\times
\bigbrak{1+R_-(\eps,\abs{\mu_1})}\;.
\end{multline}
\end{itemiz}
\end{cor}

For $\lambda_1=0$, this result reduces to Corollary~\ref{cor_cap01}. 
For negative $\lambda_1$ of order $1$, one recovers the usual Eyring--Kramers
formula, while for positive $\lambda_1$ of order $1$, one obtains the 
Eyring--Kramers formula with an extra factor $2$, which is due to the fact that
two saddles have to be crossed in a row. Note that when applying this result to
Example~\ref{ex_pitch1} near $\gamma=1/3$, all expected transition
times have to be divided by $2$, because there are two equivalent paths from
one global minimum to the other one. 


\subsection{Bifurcations with double-zero eigenvalue}
\label{ssec_zerozero}

We consider finally the case where the $\gamma$-dependent potential $V$ admits,
for $\gamma=\gamma^\star$, $z=0$ as a nonquadratic saddle with two vanishing
eigenvalues. Then the normal form is given by~\eqref{cap00_2} with $V_3=0$. This
singularity having codimension $2$, there are several different ways to perturb
it, and actually two parameters are needed to describe all of them. We restrict
our attention to the particular perturbation (which is generic, e.g., in cases
with $D_N$-symmetry, $N\geqs3$)
\begin{equation}
\label{bif001}
V(y) = \frac12 \lambda_1(\gamma)y_1^2 + \frac12
\lambda_2(\gamma)(y_2^2+y_3^2) + 
V_4(y_2,y_3;\gamma) + \frac12 \sum_{j=4}^d
\lambda_j(\gamma) y_j^2 
+ \Order{\Norm{y}^5}\;,
\end{equation}
where $V_4(y_2,y_3;\gamma)$ is of the form~\eqref{scod2-4} with positive
discriminant, satisfying $V_{1111}>0$. Here $\lambda_2(\gamma^\star)=0$, while
$\lambda_1(\gamma^\star)<0$, $\lambda_4(\gamma^\star)>0$, and so on. Again, we
no longer indicate the $\gamma$-dependence of the eigenvalues in the sequel. All
quantities except $\lambda_2$ are assumed to be bounded away from zero. 

\begin{example}
\label{example_doublezero} 
Consider the potential~\eqref{bif01} for $N\geqs3$. In Fourier
variables 
\begin{equation}
 \label{bif001A}
z_k = \frac{1}{\sqrt{N}}\sum_{j=1}^N \e^{-k\icx 2\pi/N} x_j\;,
\qquad
-\intpart{N/2} < k \leqs \intpart{N/2}\;,
\end{equation} 
the potential takes the form 
\begin{equation}
 \label{bif001B}
\widehat V(z) = \frac12 \sum_k \eta_k z_kz_{-k} 
+ \frac1{4N} \sum_{k_1+k_2+k_3+k_4=0\pmod{N}} z_{k_1}z_{k_2}z_{k_3}z_{k_4}\;,
\end{equation} 
where $\eta_k = -1+2\gamma\sin^2(k\pi/N)$. 
For $\gamma>(2\sin^2(\pi/N))^{-1}$, all $\eta_k$ except $\eta_0=-1$ are
positive, and thus the origin is a quadratic saddle. It forms the gate between
the two global minima of the potential, given (in original variables) by
$I^\pm=\pm(1,\dots,1)$. As $\gamma$ approaches $(2\sin^2(\pi/N))^{-1}$ from
above, the two eigenvalues $\eta_{\pm1}$ go to zero, and the origin becomes a
singular saddle of codimension $2$. 
In~\cite{BFG06a} it is shown that as $\gamma$ decreases further, a certain
number (which depends on $N$) of quadratic saddles emerges from the origin, all
of the same potential height. The set of all these saddles then forms the gate
between $I^-$ and $I^+$. 

The normal form of~\eqref{bif001B} is equivalent to~\eqref{bif001} with $d=N$,
up to a linear transformation (one has to work with the real and imaginary
part of each $z_k$ instead of the pairs $(z_k,z_{-k})$), and to a relabelling of
the coordinates, setting $\lambda_1=\eta_0=-1$ and
$\lambda_{2k}=\lambda_{2k+1}=\eta_k$.
\end{example}

We start by considering saddles with normal form~\eqref{bif001} for
$\lambda_2\geqs0$, when the origin is a saddle, and for slightly negative
$\lambda_2$. Indeed, in these cases one can derive a general result which does
not depend on the details of the nonlinear terms. We define $k(\ph)$, as
in~\eqref{cap00_4}, by $V_4(r\cos\ph,r\sin\ph;\gamma) = r^4 k(\ph;\gamma)$.

\begin{theorem}
\label{thm_bif1}
Assume the normal form at $z=0$ satisfies~\eqref{bif001}. Let $A$ and
$B$ belong to different \pathconnected\ components of $\OV{z}$ (respectively of
the newly created saddles if $\lambda_2<0$). Assume further that the gate 
$G(A,B)$ between $A$ and $B$ is formed by the saddle $z$ in case
$\lambda_2\geqs0$, and by the newly created saddles, otherwise. 
Then for $\lambda_2\geqs0$, 
\begin{multline}
\label{bif002}
\capacity_A(B) = \frac1{2\pi} \int_0^{2\pi} 
\sqrt{\frac{(2\pi)^{d-2}\abs{\lambda_1}}
{\bigbrak{\lambda_2+(2\eps k(\ph))^{1/2}\,}^2\lambda_4\dots\lambda_d}}\,
\Theta_+\biggpar{\frac{\lambda_2}{(2\eps k(\ph))^{1/2}}}\,\6\ph \\ 
\times \eps^{d/2} \e^{-V(z)/\eps}
\bigbrak{1+R_+(\eps,\lambda_2)}\;,
\end{multline}
while for $-(\eps\abs{\log\eps})^{1/2}\leqs\lambda_2<0$
\begin{multline}
\label{bif003}
\capacity_A(B) = \frac1{2\pi} \int_0^{2\pi} 
\sqrt{\frac{(2\pi)^{d-2}\abs{\lambda_1}}
{\lambda_4\dots\lambda_d}}\,
\frac{\Theta_-(-\lambda_2/(2\eps k(\ph))^{1/2})}
{(2\eps k(\ph))^{1/2}}\e^{\lambda_2^2/16\eps k(\ph)}\,\6\ph \\ 
\times \eps^{d/2} \e^{-V(z)/\eps}
\bigbrak{1+R_-(\eps,\lambda_2)}\;,
\end{multline}
The functions $\Theta_+$ and $\Theta_-$ are bounded above and below uniformly on
$\R_+$. They are given by 
\begin{align}
\label{bif004a}
\Theta_+(\alpha) &= \sqrt{\frac{\pi}{2}} (1+\alpha) \e^{\alpha^2/8}
\Phi\biggpar{-\frac\alpha2}\;, \\
\label{bif004b}
\Theta_-(\alpha) &= \sqrt{\frac{\pi}{2}} \Phi\biggpar{\frac\alpha2}\;,
\end{align}
where $\Phi(x)=(2\pi)^{-1/2}\int_{-\infty}^x \e^{-y^2/2}\6y$ denotes the
distribution function of the standard normal law. In particular,
\begin{equation}
\label{bif005}
\lim_{\alpha\to+\infty} \Theta_+(\alpha) = 1\;,
\qquad
\lim_{\alpha\to+\infty} \Theta_-(\alpha) = \sqrt{\frac{\pi}{2}}\;,
\end{equation}
and 
\begin{equation}
\label{bif005B}
\lim_{\alpha\to0} \Theta_+(\alpha) = \lim_{\alpha\to0} \Theta_-(\alpha) 
= \sqrt{\frac{\pi}{8}} \simeq 0.6267\;.
\end{equation}
Finally, the error terms satisfy 
\begin{equation}
\label{bif006}
\bigabs{R_\pm(\eps,\lambda)} \leqs C 
\biggbrak{
\frac{\eps\abs{\log\eps}^{3}}
{\max\set{\abs{\lambda},(\eps\abs{\log\eps})^{1/2}}}}^{1/2}
\;.
\end{equation}
\end{theorem}

\begin{figure}
\centerline{\includegraphics*[clip=true,width=75mm]{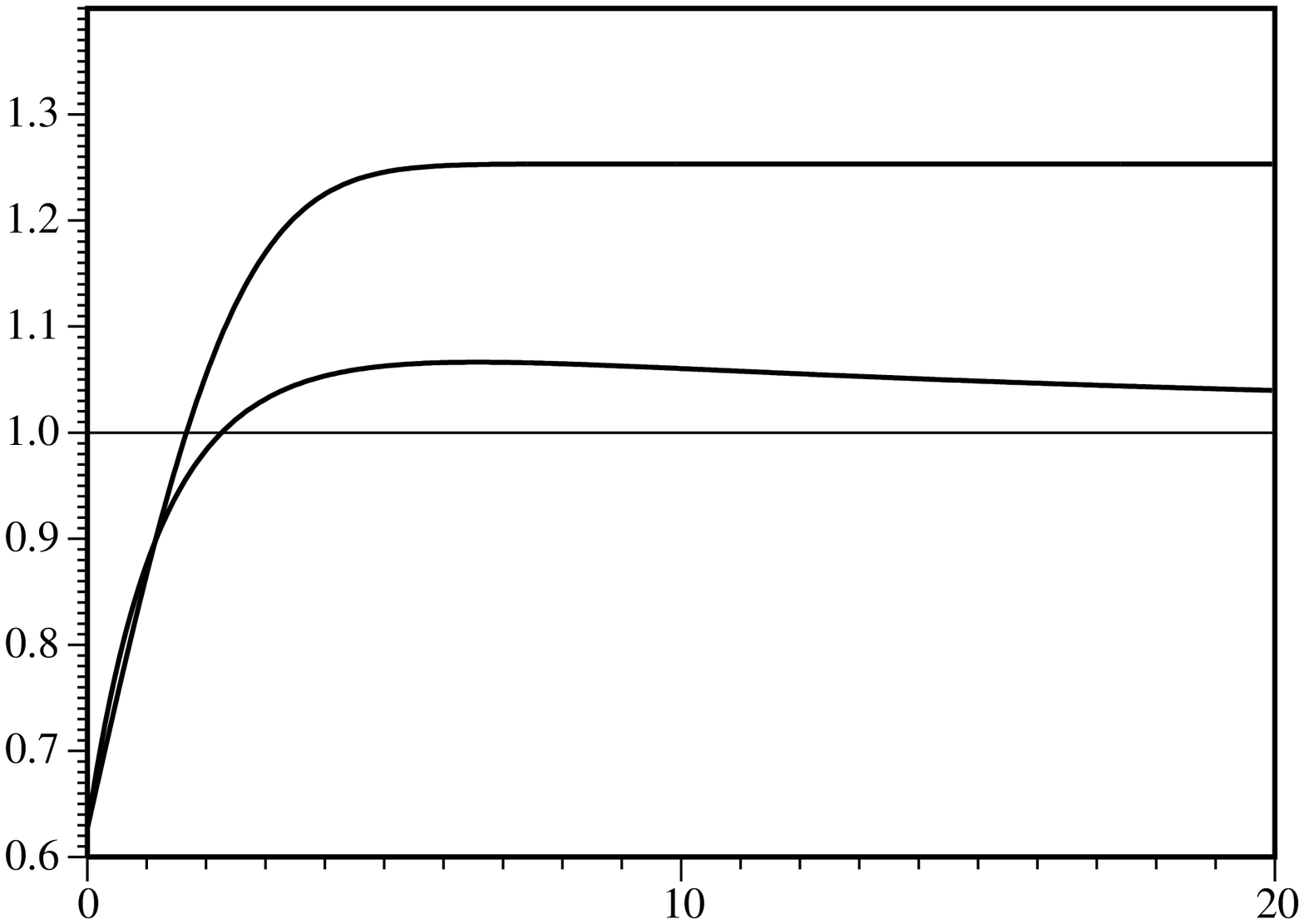}
\includegraphics*[clip=true,width=75mm]{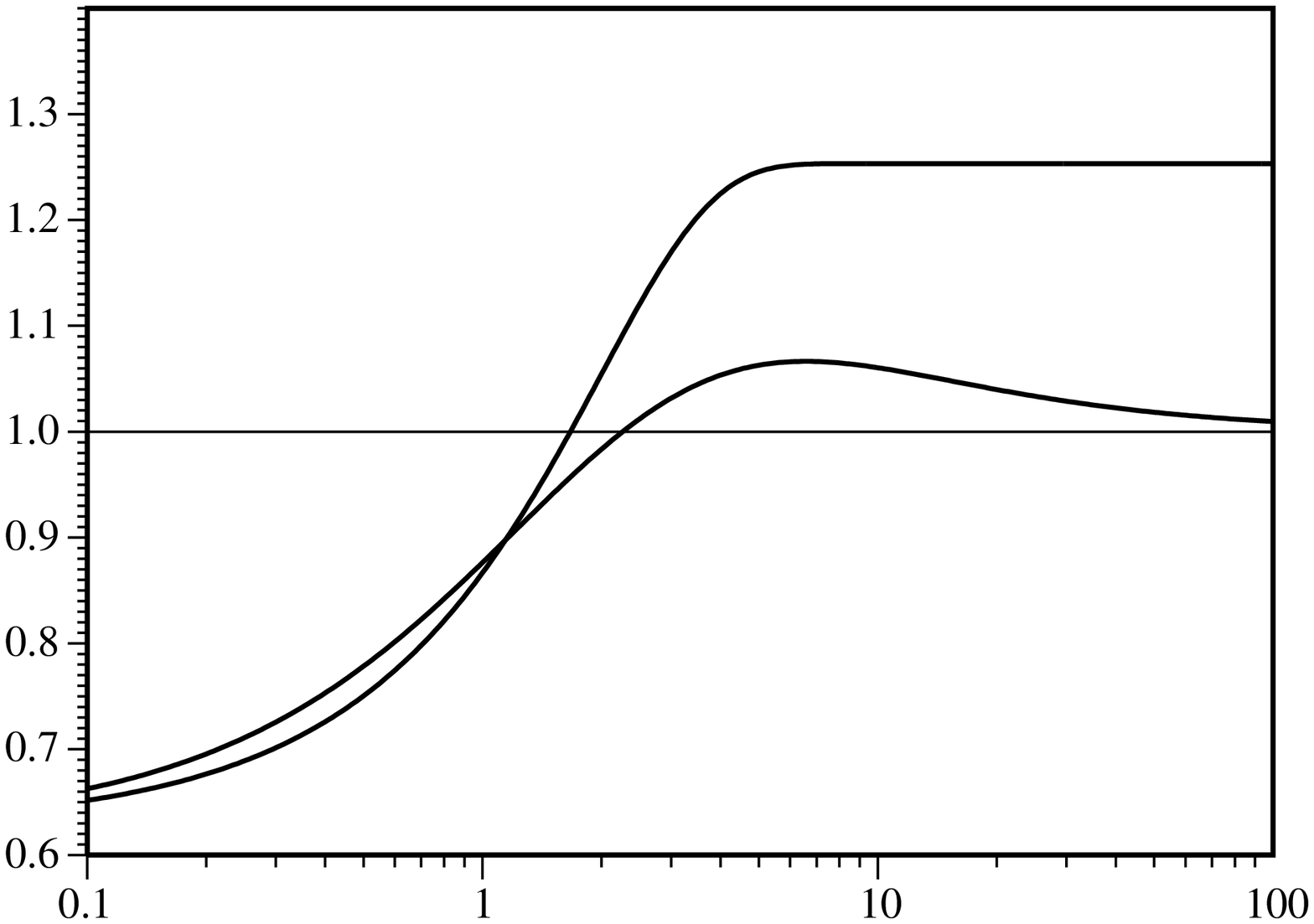}}
 \figtext{ 
	\writefig	3.5	3.9	$\Theta_+(\alpha)$
	\writefig	4.4	5.0	$\Theta_-(\alpha)$
	\writefig	12.0	3.9	$\Theta_+(\alpha)$
	\writefig	12.3	5.0	$\Theta_-(\alpha)$
 }
\caption[]{The functions $\Theta_\pm(\alpha)$, shown on a linear and on a
logarithmic scale.}
\label{fig_Theta}
\end{figure}

The proof is given in Section~\eqref{ssec_bifs}. The functions
$\Theta_\pm(\alpha)$, which are shown in~\figref{fig_Theta}, are again
universal in the sense that they will be the same for all bifurcations
admitting the normal form~\eqref{bif001}.

\goodbreak
\begin{cor}
\label{cor_bif00}
Assume the normal form at $z=0$ satisfies~\eqref{bif001}. 
\begin{itemiz}
\item	If $\lambda_2>0$, assume that the minimum of\/ $V$ in one of the
\pathconnected\ components of\/ $\OV{z}$ is reached at a unique point
$x$, which is quadratic. Let $B$ belong to a different \pathconnected\ component
of\/ $\OV{z}$, with $G(\set{x},B)=\set{z}$. Then the expected first-hitting time
of $B$ satisfies 
\begin{equation}
\label{bif007}
\bigexpecin{x}{\tau_B} = 
2\pi \sqrt{\frac{\lambda_4\dots\lambda_d}
{\abs{\lambda_1}\det(\hessian{V}(x))}}
\frac{\e^{[V(z)-V(x)]/\eps}}
{\displaystyle
\frac{1}{2\pi} \int_0^{2\pi}\frac{\Theta_+(\lambda_2/(2\eps k(\ph))^{1/2})}
{\lambda_2+(2\eps k(\ph))^{1/2}}\6\ph}
\bigbrak{1+R_+(\eps,\lambda_2)}\;.
\end{equation}
 
\item	The above situation extends to slightly negative $\lambda_2$,
that is, $-\sqrt{\eps\abs{\log\eps}}\leqs\lambda_2\leqs0$, where 
\begin{multline}
\label{bif008}
\bigexpecin{x}{\tau_B} = 
2\pi \sqrt{\frac{\lambda_4\dots\lambda_d}
{\abs{\lambda_1}\det(\hessian{V}(x))}}
\frac{\e^{[V(z)-V(x)]/\eps}}
{\displaystyle
\frac{1}{2\pi} \int_0^{2\pi}\frac{\Theta_-(\lambda_2/(2\eps k(\ph))^{1/2}} 
{(2\eps k(\ph))^{1/2}}\e^{\lambda_2^2/16\eps k(\ph)}\6\ph} \\
\times
\bigbrak{1+R_-(\eps,\lambda_2)}\;.
\end{multline}
\end{itemiz}
\end{cor}

For positive $\lambda_2$ of order $1$, one recovers the usual Eyring--Kramers
law (recall that $\lambda_2=\lambda_3$), while for $\lambda_2=0$, one recovers
Corollary~\ref{cor_cap001}. For negative $\lambda_2$, if the function $k(\ph)$
is nonconstant, the integral over $\ph$ in~\eqref{bif008} can be evaluated by
the Laplace method. Extrapolating the result to $\lambda_2$ of order $-1$, 
the integral would yield an extra $\eps^{1/2}$ which cancels with the
$\eps^{1/2}$ in the integral's denominator.

\begin{example}
\label{ex_bif00-1}
Consider the potential~\eqref{bif001B} for $N=3$ or $N\geqs5$. The resonant
terms near the bifurcation occurring for $\lambda_1=0$ are the monomials
proportional to $z_1z_{-1}$ and to $z_1^2z_{-1}^2$. In order to obtain the
standard normal form~\eqref{bif001}, we set 
$z_{\pm 1}=(y_2 \pm \icx y_3)/\sqrt{2}$ (the factor $\sqrt{2}$ guarantees that
the change
of variables is isometric). This yields 
\begin{equation}
 \label{bif0010}
V_4(y_2,y_3) = \frac{3}{8N} (y_2^2+y_3^2)^2\;, 
\end{equation}
and thus a constant $k(\ph)=3/8N$. The integrals in~\eqref{bif007}
and~\eqref{bif008} can thus easily be computed. For instance, for odd $N$ and
$\lambda_2 = -1+2\gamma\sin^2(\pi/N) > 0$, the prefactor of the transition
time from $I^-$ to a neighbourhood of $I^+$ is given by 
\begin{equation}
 \label{bif0010A}
2\pi
\frac{\lambda_4\lambda_6\dots\lambda_{N-1}}{\sqrt{\det\hessian V(I^-)}}
\frac{\lambda_2+(3\eps/4N)^{1/2}}{\Theta_+(\lambda_2/(3\eps/4N)^{1/2})}
\bigbrak{1+R_+(\eps,\lambda_2)}\;,
\end{equation} 
which admits a limit as $\lambda_2\to 0_+$. Note that the eigenvalues of
$\hessian V(I^-)$ are of the form
$\nu_1=2$, $\nu_{2k}=\nu_{2k+1}=2+2\gamma\sin^2(k\pi/N)$, so that all
quantities in~\eqref{bif0010A} are known.
\end{example}

\begin{example}
\label{ex_bif00-2}
Consider now the potential~\eqref{bif001B} for $N=4$. Then there are two
additional resonant terms, namely the monomials proportional to $z_1^4$ and
$z_{-1}^4$. Proceding as in the previous example, this yields 
\begin{equation}
 \label{bif0011}
V_4(y_2,y_3) = \frac{1}{8} (y_2^4+y_3^4)\;, 
\end{equation}
and thus $k(\ph)=(3+\cos(4\ph))/32$. It is, however, more convenient to keep
rectangular coordinates instead of using polar coordinates. Then the
coordinates $y_2$ and $y_3$ separate, each one undergoing independently a
pitchfork bifurcation. 
\end{example}

The behaviour for negative $\lambda_2 <-\sqrt{\eps\abs{\log\eps}}$ is determined
by the shape of the potential in the
$(y_2,y_3)$-plane. Near the bifurcation point, the potential is sombrero-shaped.
As $\lambda_2$ decreases, however, the sombrero may develop \lq\lq dips in the
rim\rq\rq, and these dips will determine the value of the integral over $\ph$.
We conclude the discussion by an in-depth study of this phenomenon in the case
of our model potential. 

\begin{example}
\label{ex_bif00-3}
Consider again the potential~\eqref{bif001B} for $N=3$ or $N\geqs5$. As seen in
Example~\ref{ex_bif00-1} above, the quartic term then is rotation-invariant.
This, however, is not sufficient to determine the behaviour for
$\lambda_2<-\sqrt{\eps\abs{\log\eps}}$. 

In fact, we know from symmetry arguments~\cite{BFG06a} that the
normal form of the potential around the origin has the form 
\begin{equation}
 \label{N5_1}
V(y) = -\frac12\abs{\lambda_1}y_1^2 + \frac12\lambda_2 r^2 
+ \sum_{q=2}^M C_{2q}r^{2q} + D_{2M}r^{2M}\cos(2M\ph) 
+ \frac12\sum_{j=4}^N\lambda_j y_j^2 
+ \Order{\Norm{y}^{2M+2}}\;,
\end{equation} 
with $M=N$ if $N$ is odd, and $M=N/2$ if $N$ is even. Here $(r,\ph)$ denote
polar coordinates for $(y_2,y_3)$. To be precise,
the approach presented in~\cite{BFG06a} is based on a centre-manifold analysis,
but the expression of the potential on the centre manifold is equal to the
resonant part of the normal form (this is because after the normal-form
transformation, the centre manifold is $\Norm{y}^{2M+2}$-close to the
$(y_2,y_3)$-plane).

For $\lambda_2<0$, $V(y)$ has nontrivial stationary points of the form
$z^{\star} = (0,r^\star,\ph^\star,0,\dots,0)+\Order{\lambda_{2}^{2}}$, where 
\begin{equation}
 \label{N5_2}
r^\star = \sqrt{\frac{-\lambda_2}{4C_4}} + \Order{\abs{\lambda_2}^{3/2}}\;,
\qquad
\sin(2M\ph^\star)=0\;.
\end{equation} 
The eigenvalues of the Hessian around these points are of the form 
\begin{align}
\nonumber
\mu_1 &= \lambda_1 + \Order{\abs{\lambda_2}}\;,\\
\nonumber
\mu_2 &= \frac1{r^2} \dpar{^2V}{\ph^2}(r^\star,\ph^\star) 
= \pm (2M)^2 D_{2M} \biggpar{\frac{-\lambda_2}{4C_4}}^{M-1} 
+ \Order{\abs{\lambda_2}^M}\;,\\
\nonumber
\mu_3 &= \dpar{^2V}{r^2}(r^\star,\ph^\star) 
= -2\lambda_2 + \Order{\abs{\lambda_2}^2}\;,\\
\mu_k &= \lambda_k + \Order{\abs{\lambda_2}}\;,
\qquad k= 4,\dots,N\;.
\label{N5_3} 
\end{align}
The points $z^\star$ with $\mu_2>0$ are saddles. They correspond to the \lq\lq
dips on the rim of the sombrero\rq\rq. In the sequel we will therefore assume
$\mu_2>0$.

Defining the sets $A$ and $B$ as usual, an analogous computation to the
previous ones (see Section~\ref{ssec_bifs} for details) shows that 
\begin{multline}
\label{N5_9}
\capacity_{A}(B) = 2M 
\sqrt{\frac{(2\pi)^{d-2}\abs{\mu_1}}
{(\mu_2\mu_3+(2M)^28\eps C_4)\mu_4\dots\mu_N}}
\,
\Theta_-\biggpar{\frac{\mu_3}{(8\eps C_4)^{1/2}}} 
\chi\biggpar{\frac{\mu_2\mu_3}{(2M)^2 8\eps C_4}} \\
\times 
\eps^{d/2} \e^{-V(z^\star)/\eps}
\bigbrak{1+R_-(\eps,\mu_2)}\;,
\end{multline}
with $R_{-}(\eps,\mu_2)$ as in \eqref{bif006}. Here  
$\Theta_-$ is the function defined in~\eqref{bif004b}, and 
\begin{equation}
\label{N5_10}
\chi(\alpha) = 2 \sqrt{1+\alpha}\e^{-\alpha} I_0(\alpha)\;,
\end{equation} 
where $I_0$ is the modified Bessel function of the first kind. It satisfies 
\begin{equation}
 \label{N5_10B}
\lim_{\alpha\to0} \chi(\alpha) = 2\;, \qquad
\lim_{\alpha\to\infty} \chi(\alpha) = \sqrt{\frac{2}{\pi}}\;.
\end{equation} 

\begin{figure}
\centerline{\includegraphics*[clip=true,width=75mm]{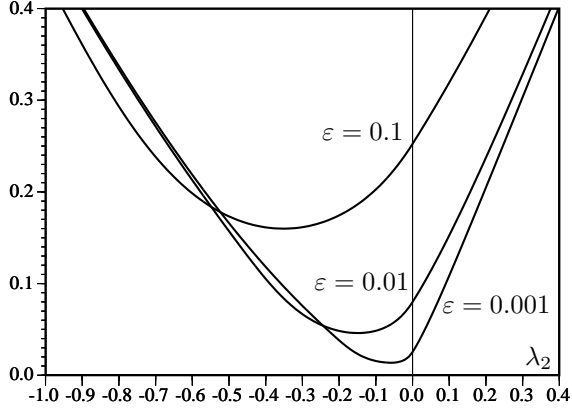}}
 \figtext{ 
	\writefig	10.5	1.0	$\lambda_2$
	\writefig	7.8	4.0	$\eps=0.1$
	\writefig	7.7	2.0	$\eps=0.01$
	\writefig	9.4	1.7	$\eps=0.001$
 }
\caption[]{$\lambda_2$-dependence of the prefactor of the expected transition
time for the model system $\eqref{bif01}$ for $N=3$.}
\label{fig_Theta2}
\end{figure}

It follows that if $B$ is a small neighbourhood of $I^+$, then for
$\lambda_2<0$, 
\begin{multline}
\label{N5_11}
\bigexpecin{I^-}{\tau_B} = 
\frac{2\pi}{2M} \sqrt{\frac{(\mu_2\mu_3+(2M)^28\eps C_4)\mu_4\dots\mu_N}
{\abs{\mu_1}\det(\hessian{V}(I^-))}}
\frac{\e^{[V(z^\star)-V(x)]/\eps}}
{\displaystyle
\Theta_-\biggpar{\frac{\mu_3}{(8\eps C_4)^{1/2}}} 
\chi\biggpar{\frac{\mu_2\mu_3}{(2M)^2 8\eps C_4}}} \\
\times
\bigbrak{1+R_-(\eps,\mu_2)}\;.
\end{multline}
The normal-form coefficient $C_4$ is in fact equal to $3/8N$
(compare~\eqref{bif0010}), which allows to check the continuity of the
prefactor at $\lambda_2=0$. 
One can identify three different regimes, depending on the value of $\mu_3>0$
(or $\lambda_2<0$): 
\begin{itemiz}
\item	For $0\leqs \mu_3 \ll \sqrt{\eps}$, the arguments of the functions
$\Theta_-$ and $\chi$ are negligible, so that the denominator can be
approximated by $\Theta_-(0)\chi(0)=\sqrt{\pi/2}$. Furthermore, the term
$(2M)^28\eps C_4$ under the square root dominates $\mu_2\mu_3$. This results in
a prefactor of order $\eps^{1/2}$, compatible with~\eqref{bif008}
and~\eqref{bif0010A}. In this regime, the potential in the $(y_2,y_3)$-plane is
almost flat. 

\item	For $\sqrt{\eps} \ll \mu_3 \ll \eps^{1/M}$, the argument of $\Theta_-$
is very large, so that $\Theta_-$ can be approximated by $\sqrt{\pi/2}$, while
the argument of $\chi$ is still negligible. The product of $\Theta_-$ and
$\chi$ is thus close to $\sqrt{2\pi}$. In this regime, the potential in the
$(y_2,y_3)$-plane is sombrero-shaped, and very close to rotation-invariant. 

\item	Finally, for $\mu_3\gg \eps^{1/M}$, the arguments of both $\Theta_-$ and
$\chi$ are very large, so that their product is close to $1$. Furthermore,
$\mu_2\mu_3$ under the square root dominates $(2M)^28\eps C_4$. In this regime,
the potential in the $(y_2,y_3)$-plane is sombrero-shaped, but with $2M$ dips of
noticeable depth. One recovers the standard Eyring--Kramers formula, with an
extra prefactor of $1/2M$ accounting for the fact that the gate consists of $2M$
saddles at equal height.  
\end{itemiz}

The $\lambda_2$-dependence of the prefactor is shown in~\figref{fig_Theta2} in
the case $N=3$. In general, the prefactor behaves like $\lambda_2$ for
$\lambda_2\gg\sqrt{\eps}$ and like $\abs{\lambda_2}^{M/2}$ for
$\lambda_2\ll-\eps^{1/M}$. 
Note that as $N$ increases, the second crossover at $\eps^{1/M}$ occurs later
and later, as the system becomes closer and closer to being rotation-invariant. 
\end{example}



\newpage
\section{Proofs}
\label{sec_proofs}


\subsection{Upper bound on the capacity}
\label{ssec_upperbnd}


We assume that the potential $V$ is of class $\cC^{r+1}$ for some $r\geqs2$, 
and that the origin $0$ is a stationary point with eigenvalues satisfying 
$\lambda_1\leqs0<\lambda_{q+1}\leqs\dots\leqs\lambda_d$ for some $q\geqs2$.
The eigenvalues $\lambda_2, \dots, \lambda_q$ lie between $\lambda_1$ and
$\lambda_{q+1}$, but may have an arbitrary sign. We assume that
$\lambda_{q+1}\leqs\dots\leqs\lambda_d$ are of order~$1$.
In the vicinity of the origin,
$V$ admits a normal form 
\begin{equation}
\label{main0}
V(y) = -u_1(y_1) + u_2(y_2,\dots,y_q) + \frac12 \sum_{j=q+1}^d \lambda_j y_j^2 
+ \Order{\Norm{y}^{r+1}}\;,
\end{equation}
where we may assume that $u_1$ and $u_2$ are polynomials of degree less or
equal $r$, and increasing at infinity. Suppose that 
\begin{itemiz}
\item
either the origin $z^{\star}=0$ itself is a saddle, 
\item
or that there exists a finite number of saddles $z^{\star}_{i}$ in the vicinity
of $z=0$. 
\end{itemiz}
Let $A$ and $B$ belong to two different \pathconnected\ components of
$\OV{z^{\star}}$ or $\OV{z^{\star}_{i}}$, respectively; and assume for
simplicity that the gate $G(A,B)$ consists of the above-mentioned saddles only.
The assumptions below will guarantee that the saddles $z^{\star}_{i}$ are close
to $z=0$. However, the results can easily be extended to more complicated
situations with several gates which are not necessarily close to the origin,
simply by summing the contributions of all gates.

\begin{prop}[Upper bound]
\label{prop_main1-1}
Assume there exist strictly positive numbers $\delta_1=\delta_1(\eps)$, 
$\delta_2=\delta_2(\eps)$ and $c\geqs0$ (independent of $\eps$) such that 
\begin{align}
\nonumber
u_1(y_1) &\leqs d\eps\abs{\log\eps} 
&& \text{whenever $\abs{y_1}\leqs\delta_1$\;,} \\
\nonumber
u_2(y_2,\dots,y_q) &\geqs -cd\eps\abs{\log\eps} 
&& \text{whenever $\Norm{(y_2,\dots,y_q)}\leqs\delta_2$\;,} \\
u_2(y_2,\dots,y_q) &\geqs 2d\eps\abs{\log\eps} 
&& \text{whenever $\Norm{(y_2,\dots,y_q)}\geqs\delta_2$\;,} 
\label{main1}
\end{align}
and such that 
\begin{equation}
\label{main1A}
\bigbrak{\delta_1(\eps)+\delta_2(\eps)}^{r+1} = \order{\eps\abs{\log\eps}}\;.
\end{equation}
Then 
\begin{equation}
\label{main1B}
\capacity_A(B) \leqs   
\eps \, \frac
{\displaystyle\int_{\cB_{\delta_2}(0)} \e^{-u_2(y_2,\dots,y_q)/\eps}
\,\6y_2\dots\6y_q}
{\displaystyle\int_{-\delta_1}^{\delta_1} \e^{-u_1(y_1)/\eps} \,\6y_1}
\prod_{j=q+1}^d 
\sqrt{\frac{2\pi\eps}{\lambda_j}} 
\bigbrak{1+R_1(\eps)}
+ R_2(\eps)\;,
\end{equation}
where
$\cB_{\delta_2}(0)=\setsuch{(y_2,\dots,y_q)}{y_2^2+\dots+y_q^2<\delta_2^2}$
and the error terms satisfy 
\begin{align}
\nonumber
R_1(\eps) & \leqs {}
	\begin{cases}
	C\bigbrak{\eps^{1/2}\abs{\log\eps}^{3/2} + \eps^{-1}(\delta_1^{r+1} + 
	\delta_2^{r+1}) + \delta_1 + \delta_2} 
	& \text{for $r=2$,} 
	\\
	C\bigbrak{\eps^{1/2}\abs{\log\eps}^{1/2} + \eps^{-1}(\delta_1^{r+1} + 
	\delta_2^{r+1}) + \delta_1 + \delta_2}
	& \text{for $r>2$,}
	\end{cases}
 \\
R_2(\eps) & {} \leqs C\eps^{3d/4+1}\delta_{1}^{-2} \leqs
C\eps^{d/2+3/2}\delta_{1}^{-2}
\label{main1C}
\end{align}
for some constant $C>0$.
\end{prop}
\begin{proof}
The proof is adapted from the proof of~\cite[Theorem~5.1]{BEGK}. 

Recall that the capacity can be computed as the minimal value of the Dirichlet
form
$\Phi_{(A\cup B)^c}$, cf.~\eqref{cappot4}, which involves an integration over
$x$. In the vicinity of the saddle, we carry out the normal-form transformation
of Proposition~\ref{prop_scodim1} or~\ref{prop_scodim2} in the integral. This
can always be done locally, setting $x=y+\rho\circ g(y)$, where $\rho$ is a
smooth cut-off function which is the identity in a small ball of radius
$\Delta$, and identically zero outside a larger ball of radius $2\Delta$. Inside
the smaller ball, we may thus assume that the potential is given
by~\eqref{main0}. 

Let $\delta=2\sqrt{(1+c)d\eps\abs{\log\eps}}$. We introduce a set 
\begin{equation}
\label{main1-2:2}
C_\eps = [-\delta_1,\delta_1] \times \cB_{\delta_2}(0) \times 
\prod_{j=q+1}^d [-\delta_j,\delta_j]\;,
\end{equation}
where $\delta_1$ and $\delta_2$ satisfy~\eqref{main1} and~\eqref{main1A} 
and we choose 
\begin{equation}
\label{main1-2:3}
\delta_j = \frac{\delta}{\sqrt{\lambda_j}}\;,
\quad
j=q+1,\dots, d\;.
\end{equation}
By assumption, making $\eps$ small enough we can construct a layer 
$\cS_\eps$ of width $2\delta_1$, separating the connected components of the
open valley $\OV{0}$, and such that $V(y)$ is strictly positive for all
$y\in\cS_\eps\setminus C_\eps$. Let $D_-$ and $D_+$ denote the connected
components of $\R^d\setminus\cS_\eps$ containing $A$ and $B$
respectively. Note that we can choose a radius $\Delta$ for the ball
$B_{\Delta}(0)$ in which we carry out the normal form transformation
independently of $\eps$ in such a way that $V$ is bounded away from zero on
$\cS_{\eps} \cap B_{\Delta}(0)^{\text{c}}$.

The variational principle~\eqref{cappot4}
implies that it is sufficient to construct a function
$h^+\in\cH_{A,B}$ such that $\Phi_{(A\cup B)^c}(h^+)$ satisfies the upper
bound. We choose 
\begin{equation}
\label{main1-2:4}
h^+(y) = 
\begin{cases}
1 & \text{for $y\in D_-$\;,}\\
0 & \text{for $y\in D_+$\;,}\\
f(y_1) & \text{for $y\in C_\eps$\;,}
\end{cases}
\end{equation}
while $h^+(y)$ is arbitrary for $y\in\cS_\eps\setminus C_\eps$, except that we
require $\Norm{\nabla h^+} \leqs \const/\delta_1$. The function $f(y_1)$ is
chosen as the solution of the one-dimensional differential equation  
\begin{equation}
\label{main1-2:5}
\eps f''(y_1) - \dpar V{y_1}(y_1,0,\dots,0) f'(y_1) = 0
\end{equation}
with boundary conditions $1$ in $-\delta_1$ and $0$ in $\delta_1$, that is,
\begin{equation}
\label{main1-2:6}
f(y_1) = 
\frac{\displaystyle \int_{y_1}^{\delta_1} \e^{V(t,0,\dots,0)/\eps} \6t}
{\displaystyle \int_{-\delta_1}^{\delta_1} \e^{V(t,0,\dots,0)/\eps} \6t}\;.
\end{equation}
Inserting $h^+$ into the expression~\eqref{cappot3} of the capacity, we obtain
two non-vanishing terms, namely the integrals over $\cS_\eps\setminus
C_\eps$ and over $C_\eps$. The first of these can be bounded  as follows. 
For sufficiently small $\Delta$,  Assumptions~\eqref{main1} and~\eqref{main1A}
imply that for all $y\in(\cS_\eps\setminus C_\eps)\cap B_{\Delta}(0)$,
\begin{align}
\nonumber
\frac{V(y)}{\eps}
&\geqs -d\abs{\log\eps} - cd\abs{\log\eps} + 2(1+c)d \abs{\log\eps} +
\bigOrder{\eps^{-1}\bigbrak{\delta+\delta_1+\delta_2}^{r+1}} \\
&\geqs \frac34 d \abs{\log\eps}
\label{main1-2:6B}
\end{align}
for sufficiently small $\eps$. On $(\cS_\eps\setminus C_\eps)\cap
B_{\Delta}(0)^{\text{c}}$, $V$ is bounded away from zero, and by the assumption
of exponentially tight level sets, the contribution from this part of the
integral is negligible. It follows that 
\begin{equation}
\label{main1-2:7}
\eps \int_{\cS_\eps\setminus C_\eps} \e^{-V(y)/\eps}
\frac{\const}{\delta_1^2} \,\6y
= \bigOrder{\eps^{3d/4+1}\delta_1^{-2}}
\bydef R_2(\eps)\;.
\end{equation}

When calculating the second term which is given by 
\begin{equation}
\Phi_{C_\eps}(h^+) 
= \eps 
\int_{C_\eps} \e^{-V(y)/\eps} \abs{f'(y_1)}^2 \,\6y 
= \eps 
\frac{\displaystyle \int_{C_\eps}\e^{-V(y)/\eps}\e^{2V(y_1,0,\dots,0)/\eps}
\,\6y}
{\biggpar{\displaystyle
\int_{-\delta_1}^{\delta_1}\e^{V(y_1,0,\dots,0)/\eps}\,\6y_1}^2}\;,
\label{main1-2:8}
\end{equation}
it is sufficient to carry out the normal form transformation in a smaller ball
of radius $2\tilde\Delta$ with $\tilde\Delta = \sqrt{d} \max\set{\delta,
\delta_1, \delta_2}$ as $B_{\tilde\Delta}(0)$ contains $C_{\eps}$. The Jacobian
of the transformation thus yields a multiplicative error term
$1+\Order{\tilde\Delta}$.
By~\eqref{main0}, we have for $y\in C_\eps$ 
\begin{equation}
\label{main1-2:9}
V(y) - 2 V(y_1,0,\dots,0) 
= u_1(y_1) 
+ u_2(y_2,\dots,y_q) 
+ \frac12 \sum_{j=q+1}^d \lambda_j y_j^2 
+ \bigOrder{\brak{\delta+\delta_1+\delta_2}^{r+1}}\;.
\end{equation}
Hence the numerator in~\eqref{main1-2:8} is given by 
\begin{multline}
\label{main1-2:10}
\int_{-\delta_1}^{\delta_1} \e^{-u_1(y_1)/\eps} \,\6y_1 
\int_{\cB_{\delta_2}(0)} \e^{-u_2(y_2,\dots,y_q)/\eps} \,\6y_2\dots\6y_q
\prod_{j=q+1}^d 
\int_{-\delta_j}^{\delta_j} \e^{-\lambda_jy_j^2/2\eps} \,\6y_j \\
\times 
\biggbrak{1+\biggOrder{\frac{\brak{\delta+\delta_1+\delta_2}^{r+1}}{\eps}}}\;.
\end{multline}
Substituting in~\eqref{main1-2:8} we get 
\begin{multline}
\label{main1-2:11}
\Phi_{C_\eps}(h^+) = \eps \,
\frac{\displaystyle
\int_{\cB_{\delta_2}(0)} \e^{-u_2(y_2,\dots,y_q)/\eps} \,\6y_2\dots\6y_q}
{\displaystyle
\int_{-\delta_1}^{\delta_1} \e^{-u_1(y_1)/\eps} \,\6y_1 }
\prod_{j=q+1}^d 
\int_{-\delta_j}^{\delta_j} \e^{-\lambda_jy_j^2/2\eps} \,\6y_j \\
\times
\biggbrak{1+\biggOrder{\frac{\brak{\delta+\delta_1+\delta_2}^{r+1}}{\eps}}}\;.
\end{multline}
Using the fact that the Gaussian integrals over $y_j$, $j=q+1,\dots,d$, are 
bounded above by $\sqrt{2\pi\eps/\lambda_j}$, the desired bound~\eqref{main1B}
follows.
\end{proof}

\begin{remark}
\label{rem_upperbnd}
Using the conditions~\eqref{main1} in order to bound the integrals over $y_1$
and $(y_2,\dots,y_q)$, one obtains as a rough {\it a priori\/} bound 
\begin{equation}
\label{main_upper_1}
\capacity_{A}(B) \leqs \const \frac{\delta_2^{q-1}}{\delta_1}
\eps^{1-q/2-(c+1/2)d}[1+R_{1}(\eps)]+R_{2}(\eps)\;.
\end{equation}
In applications we will of course obtain much sharper bounds by using explicit
expressions for $u_1(y_1)$ and $u_2(y_2,\dots,y_q)$, but the above rough bound
will be sufficient to obtain a lower bound on the capacity, valid
without further knowledge of the functions $u_1$ and $u_2$. 
\end{remark}


\subsection{Lower bound on the capacity}
\label{ssec_lowerbnd}


Before we proceed to deriving a lower bound on the capacity, we need a crude
bound on the equilibrium potential $h_{A,B}$. We obtain such a bound by adapting
similar results from~\cite[Section~4]{BEGK} to the present situation.

\begin{lemma}
\label{lem_apriori}
Let $A$ and $B$ be disjoint sets, and let $x\in(A\cup B)^c$ be such that the
ball $\cB_{\eps}(x)$ does not intersect $A\cup B$. Then there exists a constant
$C$ such that 
\begin{equation}
\label{apriori1}
h_{A,B}(x) \leqs C \eps^{-d} \capacity_{\cB_\eps(x)}(A)
\e^{\Vbar(\set{x},B)/\eps}\;.
\end{equation}
\end{lemma}

\begin{proof}
\cite[Proposition~4.3]{BEGK} provides the upper bound 
\begin{equation}
\label{apriori2:1}
h_{A,B}(x) \leqs C
\frac{\capacity_{\cB_\eps(x)}(A)}{\capacity_{\cB_\eps(x)}(B)}\;,
\end{equation}
so that it suffices to obtain a lower bound for the denominator. This is done
as in \cite[Proposition~4.7]{BEGK} with $\rho=\eps$, cf. in particular
Equation~(4.26) in that work, which provides a lower bound for the capacity in
terms of an integral of $\e^{V/\eps}$ over a critical path from $x$ to $B$.
Evaluating the integral by the Laplace method, one gets
$\e^{\Vbar(\set{x},B)/\eps}$ as leading term, with a multiplicative correction. 
The only difference is that while Bovier {\it et
al} assume quadratic saddles, which yields a correction of order 
$\sqrt{\eps}$, here we do not assume anything on the saddles, so that in the
worst case the prefactor is constant. This yields the bound~\eqref{apriori1}.
\end{proof}

The capacity $\capacity_{\cB_\eps(x)}(A)$ behaves roughly like
$\e^{-\Vbar(\set{x},A)/\eps}$, so that the bound~\eqref{apriori1} is useful
whenever $\smash{\Vbar(\set{x},A)\gg\Vbar(\set{x},B)}$. This is the case, in
particular, when $A$ and $B$ belong to different \pathconnected\ components of
the open valley of a saddle $z$, and $x$ belongs to the same component as $B$.
If, by contrast, $x$ belongs to the same component as $A$, the symmetry
$h_{A,B}(x)=1-h_{B,A}(x)$ yields a lower bound for the equilibrium potential
which is close to~$1$.


We now consider the same situation as in Section~\ref{ssec_upperbnd}. Let
$\delta_1(\eps)$, $\delta_2(\eps)$ and $c$ be the constants introduced in
Proposition~\ref{prop_main1-1}. 

\begin{prop}[Lower bound]
\label{prop_main1-2}
Let $K > \max\set{2d(2+c),d-q}$, and assume that
$\delta_1(\eps)\geqs\eps$.\footnote{It actually suffices to have $\delta_{1}
\geqs \eps^{K/2+1/4-d/8}$ which is a very weak condition satisfied whenever
$\delta_{1} \geqs \eps^{\kappa}$ for a $\kappa>1$ depending on $c,d$ and $q$.}
Furthermore, suppose there exist strictly positive numbers
$\deltahat_1=\deltahat_1(\eps)$ and
$\deltahat_2=\deltahat_2(\eps)$ such that 
\begin{align}
\nonumber
u_1(\pm\deltahat_1) &\geqs 4K\eps\abs{\log\eps}\;, 
&&  \\
\nonumber
u_2(y_2,\dots,y_q) &\leqs K\eps\abs{\log\eps} \;,
&& \text{whenever $\Norm{(y_2,\dots,y_q)}\leqs\deltahat_2$\;,} 
\label{main2}
\end{align}
and such that 
\begin{equation}
\label{main2A}
\bigbrak{\deltahat_1(\eps)+\deltahat_2(\eps)}^{r+1} =
\order{\eps\abs{\log\eps}}\;.
\end{equation}
Then 
\begin{equation}
\label{main2B}
\capacity_A(B) \geqs   
\eps \, \frac
{\displaystyle\int_{\cB_{\deltahat_2}(0)} \e^{-u_2(y_2,\dots,y_q)/\eps}
\,\6y_2\dots\6y_q}
{\displaystyle\int_{-\deltahat_1}^{\deltahat_1} \e^{-u_1(y_1)/\eps} \,\6y_1}
\prod_{j=q+1}^d 
\sqrt{\frac{2\pi\eps}{\lambda_j}} 
\bigbrak{1-R_3(\eps)}\;,
\end{equation}
where the remainder $R_3(\eps)$ satisfies 
\begin{equation}
\label{main2C}
R_3(\eps) \leqs {}
	\begin{cases}
	C\bigbrak{\eps^{1/2}\abs{\log\eps}^{3/2} 
	+ \eps^{-1}(\deltahat_1^{r+1} +	\deltahat_2^{r+1})
	+ \deltahat_1 + \deltahat_2 + \sqrt{\eps} }
	& \text{for $r=2$,} 
	\\
	C\bigbrak{\eps^{1/2}\abs{\log\eps}^{1/2}
	+ \eps^{-1}(\deltahat_1^{r+1} +\deltahat_2^{r+1})
	+ \deltahat_1 + \deltahat_2 + \sqrt{\eps}}
	& \text{for $r>2$,} 
	\end{cases}
\end{equation}
for some constant $C>0$.
\end{prop}

\begin{proof}
As in the proof of Proposition~\ref{prop_main1-1}, we start by locally carrying
out the normal form transformation in the integral defining the Dirichlet form. 
Next we define a slightly different neighbourhood of the saddle, 
\begin{equation}
\label{main1-2:21}
\Chat_\eps = [-\deltahat_1,\deltahat_1] \times \cB_{\deltahat_2}(0)
\times \prod_{j=q+1}^d [-\deltahat_j,\deltahat_j]
= [-\deltahat_1,\deltahat_1] \times \Chat_\eps^\perp\;,
\end{equation}
where we now choose 
\begin{equation}
\label{main1-2:22}
\deltahat_j = \frac{\delta}{\sqrt{(d-q)\lambda_j}}\;,
\quad
j=q+1,\dots, d\;,
\end{equation}
with $\delta=\sqrt{K\eps\abs{\log\eps}}$.
The reason for this choice is that we want the potential to be smaller than
$-\delta^2$ on the \lq\lq sides\rq\rq\ $\set{\pm\deltahat_1}\times
\smash{\Chat_\eps^\perp}$ of the box. Indeed, we have for all $y_\perp \in
\smash{\Chat_\eps^\perp}$,
\begin{align}
\nonumber
\frac{V(\pm\deltahat_1,y_\perp)}{\eps}
&\leqs -4K\abs{\log\eps} + K\abs{\log\eps} + (d-q)\frac{\delta^2}{2\eps(d-q)}
+ \bigOrder{\eps^{-1}\brak{\delta+\deltahat_1+\deltahat_2}^{r+1}} \\
&\leqs -K\abs{\log\eps}
\label{main1-2:22B}
\end{align}
for sufficiently small $\eps$. As a consequence, 
if $h^\star=h_{A,B}$ denotes the equilibrium potential, Lemma~\ref{lem_apriori}
and the {\it a priori\/} bound~\eqref{main_upper_1} yield 
\begin{align}
\label{main1-2:22C}
h^\star(\deltahat_1,y_\perp) & {}=
\BigOrder{\eps^{-d}
\Bigbrak{\frac{\delta_{2}^{q-1}}{\delta_{1}}\eps^{1-q/2-(c+1/2)d}\brak{1+R_{1}
(\eps)}+R_{2}(\eps)}
\e^{V(\deltahat_1,y_\perp)/\eps}} \\ \nonumber
&{}= \BigOrder{\eps^{-d}
\Bigbrak{\frac{{\order{(\eps\abs{\log\eps}}}^{(q-1)/(r+1)})}{\eps^{K/2+1/4-d/8}}
\eps^{1-q/2-(c+1/2)d}\brak{1+\abs{\log\eps}}
}
\eps^{K} + \eps^{1/2}
}  \\
&= \bigOrder{\eps^{1/2}}
\end{align}
while 
\begin{equation}
\label{main1-2:22D}
h^\star(-\deltahat_1,y_\perp) = 1 - \bigOrder{\eps^{1/2}}\;.
\end{equation}
We can now proceed to deriving the lower bound. Observe that 
\begin{equation}
\label{main1-2:23}
\capacity_A(B) = \Phi_{(A\cup B)^c}(h^\star) \geqs
\Phi_{\Chat_\eps}(h^\star)\;.
\end{equation}
Now we can write, for any $h\in\cH_{A,B}$, 
\begin{align}
\nonumber
\Phi_{\Chat_\eps}(h) 
&\geqs \eps\int_{\Chat_\eps} \e^{-V(y)/\eps} 
\biggpar{\dpar h{y_1}}^2 \,\6y \\
\nonumber
&= \eps\int_{\Chat_\eps^\perp} \int_{-\deltahat_1}^{\deltahat_1} 
\e^{-V(y)/\eps} \biggpar{\dpar{h(y_1,y_\perp)}{y_1}}^2 \,\6y_1 \6y_\perp \;,
\label{main1-2:24A}
\end{align}
and thus 
\begin{equation}
\Phi_{\Chat_\eps}(h^\star)\geqs\eps\int_{\Chat_\eps^\perp} 
\biggbrak{
\inf_{f\colon f(\pm\deltahat_1)=h^\star(\pm\deltahat_1,y_\perp)}
 \int_{-\deltahat_1}^{\deltahat_1} \e^{-V(y)/\eps} f'(y_1)^2 \,\6y_1}
\6y_\perp\;.
\label{main1-2:24}
\end{equation}
The Euler--Lagrange equation for the variational problem is
\begin{equation}
\label{main1-2:25}
\eps f''(y_1) - \dpar V{y_1}(y_1,y_\perp) f'(y_1) = 0
\end{equation}
with boundary conditions $h^\star(-\deltahat_1,y_\perp)$ in $-\deltahat_1$ and
$h^\star(\deltahat_1,y_\perp)$ in $\deltahat_1$, and has the solution 
\begin{equation}
\label{main1-2:26}
f(y_1) = h^\star(\deltahat_1,y_\perp) - 
\bigbrak{h^\star(\deltahat_1,y_\perp) - h^\star(-\deltahat_1,y_\perp)}
\frac{\displaystyle \int_{y_1}^{\deltahat_1} \e^{V(t,y_\perp)/\eps} \6t}
{\displaystyle \int_{-\deltahat_1}^{\deltahat_1} \e^{V(t,y_\perp)/\eps} \6t}\;.
\end{equation}
As a consequence, 
\begin{equation}
\label{main1-2:27}
f'(y_1) = \bigbrak{h^\star(\deltahat_1,y_\perp) -
h^\star(-\deltahat_1,y_\perp)} 
\frac{\e^{V(y_1,y_\perp)/\eps}}
{\displaystyle \int_{-\deltahat_1}^{\deltahat_1} \e^{V(t,y_\perp)/\eps} \6t}\;,
\end{equation}
so that substitution in~\eqref{main1-2:24} yields 
\begin{equation}
\label{main1-2:28}
\Phi_{\Chat_\eps}(h^\star) 
\geqs \eps \int_{\Chat_\eps^\perp} 
\frac{\bigbrak{h^\star(\deltahat_1,y_\perp) -
h^\star(-\deltahat_1,y_\perp)}^2}
{\displaystyle \int_{-\deltahat_1}^{\deltahat_1} \e^{V(t,y_\perp)/\eps} \6t}
\,\6y_\perp\;.
\end{equation}
The bounds~\eqref{main1-2:22C} and \eqref{main1-2:22D} on $h^\star$ show that
the numerator is of the form $1-\Order{\eps^{1/2}}$. It now suffices to use the
normal form~\eqref{main0} of the potential, and to carry out integration with
respect to
$y_{q+1},\dots y_d$. 
\end{proof}


\subsection{Non-quadratic saddles}
\label{ssec_nonquad}

\begin{proof}[{\sc Proof of Theorem~\ref{thm_cap01}}]
For the upper bound, it suffices to apply Proposition~\ref{prop_main1-1}
in the case $r=4$, $q=2$, $u_1(y_1)=C_4y_1^4$ and
$u_2(y_2)=\lambda_2y_2^2/2$. The conditions for the upper bound are fulfilled
for $\delta_1=(d\eps\abs{\log\eps}/C_4)^{1/4}$,
$\delta_2=2(d\eps\abs{\log\eps}/\lambda_2)^{1/2}$ and $c=0$. This yields error
terms $R_1(\eps)=\Order{\eps^{1/4}\abs{\log\eps}^{5/4}}$ and
$R_2(\eps)=\Order{\eps^{d/2+1}/\abs{\log\eps}^{1/2}}$. The integrals over $y_1$
and $y_2$ can be computed explicitly as extending their bounds to $\pm\infty$
only produces a negligible error, and we see that the contribution of
$R_{2}(\eps)$ is also negligible. A matching lower bound is obtained in a
completely
analogous way, using Proposition~\ref{prop_main1-2} with $\deltahat_{1}$,
$\deltahat_{2}$ of the same order as $\delta_{1}$, $\delta_{2}$, respectively,
yielding that $R_{3}(\eps)$ is of the same order as $R_{1}(\eps)$. 
\end{proof}

\begin{proof}[{\sc Proof of Theorem~\ref{thm_cap02}}]
The proof is essentially the same as the previous one, only with the r\^oles of
$\delta_1$ and $\delta_2$ interchanged. 
\end{proof}

\begin{proof}[{\sc Proof of Theorem~\ref{thm_cap001}}]
We again apply Propositions~\ref{prop_main1-1} and~\ref{prop_main1-2}, now with
$q=3$, $u_1(y_1)=\abs{\lambda_1}y_1^2/2$ and $u_2(y_2,y_3)=V_4(y_2,y_3)$. The
conditions for the upper bound are fulfilled for
$\delta_1=(2d\eps\abs{\log\eps}/\abs{\lambda_1})^{1/2}$,
$\delta_2=(2d\eps\abs{\log\eps}/K_-)^{1/4}$ and $c=0$, and similarly for the
lower bound. This yields error terms
$R_1(\eps)=\Order{\eps^{1/4}\abs{\log\eps}^{5/4}}$ and
$R_2(\eps)=\Order{\eps^{(d+1)/2}/\abs{\log\eps}}$. The integral over $y_1$ can
be computed
explicitly. Writing the integral over $y_2$ and $y_3$ in polar coordinates, and
performing the integration over $r$ yields the stated expression.
\end{proof}



\subsection{Bifurcations}
\label{ssec_bifs}


We decompose the proof of Theorem~\ref{thm_pitch1} into several steps, dealing
with positive and negative $\lambda_{2}$ separately.

\begin{prop}
\label{prop_proof1}
Under the assumptions of Theorem~\ref{thm_pitch1}, and for $\lambda_2>0$, 
\begin{equation}
\label{proof1}
\capacity_A(B) = \frac{I_{a,\eps}}{(2C_4)^{1/4}}
\sqrt{\frac{(2\pi)^{d-3}\abs{\lambda_1}}{\lambda_3\dots\lambda_d}}
\eps^{d/2-1/2} \e^{-V(z)/\eps} \bigbrak{1+R_+(\eps)}\;,
\end{equation}
where $R_+(\eps)$ is defined in~\eqref{pitch10}, $I_{a,\eps}$ is the
integral
\begin{equation}
\label{proof2}
I_{a,\eps} = \int_{-\infty}^\infty \e^{-(x^4+a x^2)/2\eps} \6x\;,
\end{equation}
and $a=\lambda_2/\sqrt{2C_4}$.
\end{prop}
\begin{proof}
It suffices to apply Propositions~\ref{prop_main1-1} and~\ref{prop_main1-2} with
$r=4$, $q=2$, $u_{1}$ a quadratic function of $y_{1}$ and $u_{2}$ a polynomial
of degree $4$ in $y_{2}$ and $c=0$. We only need to
take some care in the choice of the $\delta_i$. The conditions yield 
$\delta_1=(2d\eps\abs{\log\eps}/\abs{\lambda_1})^{1/2}$ and 
\begin{equation}
\label{proof3}
\delta_2^2 =
\frac{-\lambda_2+\sqrt{\lambda_2^2+32dC_4\eps\abs{\log\eps}}}{4C_4}\;.
\end{equation}
For $\lambda_2>(\eps\abs{\log\eps})^{1/2}$, this implies that $\delta_2$ is of
order $(\eps\abs{\log\eps}/\lambda_2)^{1/2}$, while for
$0<\lambda_2\leqs (\eps\abs{\log\eps})^{1/2}$, it yields $\delta_2$ of order
$(\eps\abs{\log\eps})^{1/4}$. The expressions of $\deltahat_1$ and
$\deltahat_2$ are similar. This yields the stated error terms. The integral
over $y_1$ is carried out explicitly, while the integral over $y_2$ equals
$I_{a,\eps}/(2C_4)^{1/4}$, up to a negligible error term.
\end{proof}

\begin{prop}
\label{prop_proof2}
Under the assumptions of Theorem~\ref{thm_pitch1}, and for $\lambda_2<0$, 
\begin{equation}
\label{proof4}
\capacity_A(B) = \frac{J_{b,\eps}}{(2C_4)^{1/4}}
\sqrt{\frac{(2\pi)^{d-3}\abs{\mu_1}}{\mu_3\dots\mu_d}}
\eps^{d/2-1/2} \e^{-V(z_\pm)/\eps} \bigbrak{1+R_-(\eps)}\;,
\end{equation}
where $R_-(\eps)$ is defined in~\eqref{pitch10}, $J_{b,\eps}$ is the
integral
\begin{equation}
\label{proof5}
J_{b,\eps} = \int_{-\infty}^\infty \e^{-(x^2-b/4)^2/2\eps} \6x\;,
\end{equation}
and $b=\mu_2/\sqrt{2C_4}$.
\end{prop}
\begin{proof}
First note that for small negative $\lambda_2$, 
\begin{equation}
\label{proof6:1}
u_2(y_2) = C_4 \biggpar{y_2^2 - \frac{\mu_2}{8C_4}}^2 - \frac{\mu_2^2}{64 C_4}
+ \Order{\abs{\lambda_2}^{3/2}y_2^2}\;,
\end{equation}
where the constant term corresponds to $V(z_\pm)-V(z)$ (recall $z=0$). The
situation is more
difficult than before, because $u_2$ is not increasing on $\R_+$. When
applying Proposition~\ref{prop_main1-1}, we distinguish two regimes.
\begin{itemiz}
\item	For $\mu_2 < (\eps\abs{\log\eps})^{1/2}$, it is sufficient to choose
$\delta_2$ of order $(\eps\abs{\log\eps})^{1/4}$.
\item	For $\mu_2 \geqs (\eps\abs{\log\eps})^{1/2}$, we cannot apply
Proposition~\ref{prop_main1-1} as is, but first split the integral over $y_2$
into the integrals over $\R_+$ and over $\R_-$. Each integral is in fact
dominated by the integral over an interval of order
$(\eps\abs{\log\eps}/\mu_2)^{1/2}$ around the minimum
\mbox{$y_2=\pm(\mu_2/8C_4)^{1/2}$},
so that one can choose $\delta_2$ of that order. 
\end{itemiz}
We make a similar distinction between regimes when choosing $\deltahat_2$ in
order to apply Proposition~\ref{prop_main1-2}. This yields the stated error
terms, and the integrals are treated as before.
\end{proof}

\begin{proof}[{\sc Proof of Theorem~\ref{thm_pitch1}}]
In order to complete the proof of Theorem~\ref{thm_pitch1}, it remains to
examine the integrals $I_{a,\eps}$ and $J_{b,\eps}$. First note that 
\begin{equation}
\label{proof7}
I_{a,\eps} 
= \sqrt{\frac{2\pi\eps^{1/2}}{1+\alpha}} \Psi_+(\alpha)\;,
\end{equation}
where $\alpha=a/\sqrt{\eps}$ and 
\begin{equation}
\label{proof8}
\Psi_+(\alpha) = \sqrt{\frac{1+\alpha}{2\pi}} 
\int_{-\infty}^{\infty} \e^{-(y^4+\alpha y^2)/2}\6y\;.
\end{equation}
The change of variables $y=z/\sqrt{1+\alpha}$ yields 
\begin{equation}
\label{proof9}
\Psi_+(\alpha) =\frac1{\sqrt{2\pi}} \int_{-\infty}^{\infty} 
\exp\biggset{-\frac12\biggbrak{\frac{z^4}{(1+\alpha)^2} + \frac{\alpha
z^2}{1+\alpha}}}\6z\;,
\end{equation}
which allows to show that $\Psi_+$ is bounded above and below by positive
constants, and to compute the limits as $\alpha\to0$ and $\alpha\to\infty$. The
expressions in terms of Bessel functions are obtained by observing that 
\begin{equation}
\label{proof10}
f(\delta) \defby
\int_{-\infty}^{\infty} 
\exp\biggset{-\frac12\biggbrak{y^4+2\delta y^2+\frac{\delta^2}{2}}}\6y
= \sqrt{\frac\delta2}K_{1/4}\biggpar{\frac{\delta^2}{4}}\;,
\end{equation}
because it satisfies the equation $f''(\delta)=(\delta^2/4)f(\delta)$.
The other integral is treated in a similar way.
\end{proof}

\begin{proof}[{\sc Proof of Theorem~\ref{thm_lpitch1}}]
For $\lambda_1<0$, the proof is analogous to the proof of
Theorem~\ref{thm_pitch1}, with the r\^oles of $y_1$ and $y_2$ interchanged. The
same applies for positive $\lambda_1$ up to order $\sqrt{\eps\abs{\log\eps}}$. 

For larger $\lambda_1$, Propositions~\ref{prop_proof1} and~\ref{prop_proof2}
have to be slightly adapted: 
\begin{itemiz}
\item	For the upper bound, we define a neighbourhood $C^+_\eps$ of $z_+$ and 
a neighbourhood $C^-_\eps$ of $z_-$ in the usual way. Instead of two regions
$D_\pm$, we construct three regions $D_-, D_0$ and $D_+$, intersecting
respectively $O_-, O_0$ and $O_+$ and contained in the corresponding basins of
attraction (see \figref{fig_longpitch}c). They are separated by layers
$S_\eps^\pm$. The function $h_+$ is then defined to be equal to $1$ in $D_-$,
to $1/2$ in $D_0$ and to $0$ in $D_+$. Inside the boxes $C^\pm_\eps$, $h_+$ is
constructed in a similar way as before, only with different boundary
conditions. This yields a factor $1/2$ in the capacity. 

\item	For the lower bound, we first construct boxes $\widehat
C^\pm_\eps$ around the saddles $z^\pm$ in the same way as before. Then we
connect $\widehat C^+_\eps$ and $\widehat C^-_\eps$ by a tube staying inside
$O_0$ (whose cross-section is of the same size as the sides of the boxes). One
can define coordinates $(y_1,\dots,y_d)$, given by the normal-form
transformation inside the boxes, and such that $y_1$ runs along the length of
the tube, in such a way that the Jacobian of the transformation $x\mapsto y$ is
close to $1$. Lemma~\ref{lem_apriori} is still applicable, and yields {\it a
priori\/}
bounds on the equilibrium potential on the sides of the boxes not touching the
tube (that is, contained in $O_\pm$). The Dirichlet form is then bounded below
by restricting the domain of integration to the union of the boxes and the
connecting tube. The remainder of the proof is similar, except that the range
of $y_1$ is larger. The value of the integral is dominated by the contributions
of the two boxes. 
\qed
\end{itemiz}
\renewcommand{\qed}{}
\end{proof}

\begin{proof}[{\sc Proof of Theorem~\ref{thm_bif1}}]
We first consider the case $\lambda_2\geqs0$. 
It suffices to apply Propositions~\ref{prop_main1-1} and~\ref{prop_main1-2},
taking some care in the choice of the $\delta_i$. The conditions yield 
$\delta_1=(2d\eps\abs{\log\eps}/\abs{\lambda_1})^{1/2}$ and 
\begin{equation}
\label{proof003}
\delta_2^2 =
\frac{-\lambda_2+\sqrt{\lambda_2^2+32dK_-\eps\abs{\log\eps}}}{4K_-}\;.
\end{equation}
For $\lambda_2>(\eps\abs{\log\eps})^{1/2}$, this implies that $\delta_2$ has
order $(\eps\abs{\log\eps}/\lambda_2)^{1/2}$, while for
$0\leqs\lambda_2\leqs(\eps\abs{\log\eps})^{1/2}$, it yields $\delta_2$ of order
$(\eps\abs{\log\eps})^{1/4}$. The expressions of $\deltahat_1$ and
$\deltahat_2$ are similar. This yields the stated error terms. The integral
over $y_1$ is carried out explicitly, while the integral over $y_2$ and $y_3$
gives, using polar coordinates,
\begin{align}
\nonumber
\int_{\cB_{\delta_2}(0)} \e^{-u_2(y_2,y_3)/\eps} \,\6y_2\6y_3 
&= \int_0^{2\pi}\int_0^{\delta_2} \e^{-(\lambda_2r^2+2k(\ph)r^4)/2\eps}
r\,\6r\6\ph \\
&= \int_0^{2\pi} \frac{\sqrt\eps}{\sqrt{2k(\ph)}}
\int_0^{(2k(\ph)/\eps)^{1/4}\delta_2} \e^{-(y^4+\alpha(\ph)y^2)/2}
y\,\6y\6\ph\;,
\label{proof006:0}
\end{align}
where $\alpha(\ph)=\lambda_2/\sqrt{2\eps k(\ph)}$. Using the fact that
\begin{equation}
\label{proof00}
\int_0^\infty \e^{-(y^2+d)^2/2} y\,\6y 
= \frac12 \int_d^\infty \e^{-z^2/2}\,\6z 
= \sqrt{\frac{\pi}{2}}\Phi(-d)\;,
\end{equation}
a straightforward computation shows that the integral over $y$ is
approximated by 
\begin{equation}
\label{proof006:A}
 \frac{\Theta_+(\alpha(\ph))}{1+\alpha(\ph)}\;.
\end{equation} 
For small negative $\lambda_2$, we can write 
\begin{equation}
\label{proof006:1}
u_2(y_2,y_3) = 
k(\ph) \biggpar{r^2-\frac{\abs{\lambda_2}}{4k(\ph)}}^2 -
\frac{\lambda_2^2}{16k(\ph)} 
\;,
\end{equation}
where the constant term corresponds to the actual minimum of the potential. 

If $-(\eps\abs{\log\eps})^{1/2}<\lambda_2<0$, it suffices to apply
Propositions~\ref{prop_main1-1} and~\ref{prop_main1-2} with $\delta_2$ of order
$(\eps\abs{\log\eps})^{1/4}$. 
\end{proof}

\begin{proof}[{\sc Proof of~\eqref{N5_9}}]
The main task is to compute the integral related to $u_{2}$, namely
\begin{equation}
\label{proofN3:1}
\cJ \defby
\int \e^{-[\lambda_2r^2+2\sum_{q=2}^M C_{2q}r^{2q}]/2\eps} 
\int_0^{2\pi}\e^{-D_{2M}r^{2M}\cos(2M\ph)/\eps}\,\6\ph 
 r\6r\;.
\end{equation}
We first carry out the integral over $\ph$, yielding 
\begin{equation}
 \label{proofN3:2}
\int_0^{2\pi} \e^{-D_{2M}r^{2M}\cos(2M\ph)/\eps} \,\6\ph 
=  2\pi I_0\biggpar{\frac{D_{2M}r^{2M}}{\eps}}\;,
\end{equation} 
where $I_0$ is the modified Bessel function 
\begin{equation}
\label{proofN3:3}
I_0(\alpha) = \frac1{2\pi} \int_0^{2\pi} \e^{\alpha\cos\ph}\,\6\ph\;. 
\end{equation}
The Laplace method shows that for large $\alpha$, $I_0(\alpha)$ behaves like
$\e^\alpha/\sqrt{2\pi\alpha}$. We thus introduce the bounded function 
\begin{equation}
 \label{proofN3:4}
\chi(\alpha) = 2\sqrt{1+\alpha} \e^{-\alpha}I_0(\alpha)
= \frac{\sqrt{1+\alpha}}{\pi} \int_0^{2\pi} \e^{-\alpha(1-\cos\ph)}\,\6\ph\;. 
\end{equation} 
Inserting in~\eqref{proofN3:1} and performing the change of variable 
$r^{2}=(\eps/2C_4)^{1/2}u$ yields 
\begin{equation}
 \label{proofN3:5}
\cJ = 
\frac{\pi}{\sqrt{8C_4}}
\int \e^{-[u^2-2(\abs{\lambda_2}/\sqrt{8 C_4})u + \Order{u^{3}}]/2\eps}
\frac{\chi(D_{2M}\eps^{-1}(u/(2 C_4)^{1/2})^{M})}
{\sqrt{1+D_{2M}\eps^{-1}(u/(2 C_4)^{1/2})^{M}}}
\,\6u\;. 
\end{equation} 
Applying the Laplace method shows that the integral is dominated by $u$ close to
$u^{\star}=\abs{\lambda_{2}}/(8 C_{4})^{1/2}$. Relating the obtained expression
to the eigenvalues at the new saddles via the relation
\begin{equation}
 \label{proofN3:6}
\frac{D_{2M}}{\eps} \Bigpar{\frac{u^\star}{(2C_{4})^{1/2}}}^{M}
= \frac{\mu_2\mu_3}{(2M)^2} \frac{1+\Order{\lambda_2}}{8\eps C_4} 
\end{equation} 
yields the necessary control on $\cJ$. The stated formula for the capacity
follows.
\end{proof}


\appendix

\section{Normal forms}
\label{app_nf}

\begin{proof}[{\sc Proof of Proposition~\ref{prop_scodim1}}]
Let us denote by $\cG_k(n,m)$ the vector space of functions
$g\colon\R^n\to\R^m$
which are homogeneous of degree $k$ (i.e., $g(tx)=t^kg(x)\;\forall x$). We
write the Taylor series of $V$ in the form 
\begin{equation}
\label{scod1-6:1}
V(x) = V_2(x) + V_3(x) + V_4(x) + \order{\Norm{x}^4}\;,
\end{equation}
where $V_k\in\cG_k(d,1)$ for $k=2,3,4$. We first look for a function
$g_{2}\in\cG_2(d,d)$ such that $V\circ [\id+g_2]$ contains as few terms of
order $3$ as possible. The Taylor series of $V\circ [\id+g_2]$ can be
written 
\begin{multline}
\label{scod1-6:2}
V(x+g_2(x)) \\=  V_2(x) 
+ \underbrace{\nabla V_2(x) \cdot g_2(x) + V_3(x)}_{\text{order $3$}} 
+ \underbrace{V_2(g_2(x)) +  \nabla V_3(x) \cdot g_2(x) +
V_4(x)}_{\text{order $4$}} 
+ \order{\Norm{x}^4}\;. 
\end{multline}
Now consider the so-called adjoint map $T\colon\cG_2(d,d) \to \cG_3(d,1)$,
$g_2\mapsto \nabla V_2(\cdot) \cdot g_2$, seen as a linear map between
vector spaces. All terms of $V_3(x)$ in the image of $T$ can be eliminated by a
suitable choice of $g_2$. Let $e_l$ denote the $l$th vector in the canonical
basis of $\R^d$. We see that 
\begin{equation}
\label{scod1-6:3}
T(x_jx_ke_l) = \lambda_l x_jx_kx_l \neq 0 
\qquad
\text{for $l=2,\dots,d$\;.}
\end{equation}
Thus all monomials except $x_1^3$ are in the image of $T$. Since $T$ involves
multiplication by $x_2$ or $x_3$ or $\dots$ or $x_d$, however, $x_1^3$ is not
in the image of $T$. Hence this term is resonant. We can thus choose $g_2$ in
such a way that 
\begin{equation}
\label{scod1-6:4}
V(x+g_2(x)) = V_2(x) 
+ \underbrace{V_{111}x_1^3}_{\text{order $3$}} 
+ \underbrace{V_2(g_2(x)) + \nabla V_3(x) \cdot g_2(x) +
V_4(x)}_{\text{order $4$}} 
+ \order{\Norm{x}^4}\;. 
\end{equation}
Now a completely analogous argument shows that we can construct a function
$g_3\in\cG_3(d,d)$ such that $V\circ [\id+g_2]\circ [\id+g_3]$ has some
constant times $x_1^4$ as the only term of order $4$. It remains to determine
this constant. From~\eqref{scod1-6:4} we deduce that it has the expression 
\begin{equation}
\label{scod1-6:5}
C_4 = 
\frac12\sum_{j=2}^d \lambda_j (g^j_{11})^2
+ \sum_{j=1}^d V_{11j} g^j_{11}
+ V_{1111}\;,
\end{equation}
where $g^j_{11}$ denotes the coefficient of $x_1^2e_j$ in $g_2$. The
expression of $T$ shows that necessarily $g^j_{11}=-V_{11j}/\lambda_j$ for
$j=2,\dots, d$, while we may choose $g^1_{11}=0$. This yields~\eqref{scod1-4}.
\end{proof}


\small

\bibliography{../../BFG}
\bibliographystyle{amsalpha}               

\goodbreak
\bigskip\bigskip\noindent
{\small 
Nils Berglund \\ 
Universit\'e d'Orl\'eans, Laboratoire {\sc Mapmo} \\
{\sc CNRS, UMR 6628} \\
F\'ed\'eration Denis Poisson, FR 2964 \\
B\^atiment de Math\'ematiques, B.P. 6759\\
45067~Orl\'eans Cedex 2, France \\
{\it E-mail address: }{\tt nils.berglund@univ-orleans.fr}

\bigskip\noindent
Barbara Gentz \\ 
Faculty of Mathematics, University of Bielefeld \\
P.O. Box 10 01 31, 33501~Bielefeld, Germany \\
{\it E-mail address: }{\tt gentz@math.uni-bielefeld.de}

}


\end{document}